\documentclass{article}

\usepackage{arxiv}

\usepackage[utf8]{inputenc} 
\usepackage[T1]{fontenc}    
\usepackage{hyperref}       
\usepackage{url}            
\usepackage{booktabs}       
\usepackage{amsfonts}       
\usepackage{nicefrac}       
\usepackage{microtype}      
\usepackage{lipsum}		    
\usepackage{graphicx}
\usepackage[square,sort,comma,numbers]{natbib}
\usepackage{doi}
\usepackage{float}
\usepackage{caption}
\usepackage{subcaption}
\usepackage{xcolor}
\usepackage{amsmath}
\usepackage{amssymb}
\usepackage{amsthm} 

\renewcommand{\phi}{\varphi}
\newtheorem{theorem}{Theorem}
\newtheorem{proposition}[theorem]{Proposition}%

\title{Generalized high-order minimization-based polynomial corrections on unfitted spectral elements for the Poisson problem}

\author{\href{https://orcid.org/0000-0002-7514-301X}{\includegraphics[scale=0.06]{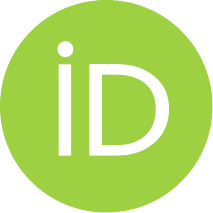}\hspace{1mm}Mirco Ciallella}$^{*\dagger}$ \\
	Laboratoire Jacques-Louis Lions\\
	Université Paris Cité\\
	75013, Paris \\
	\texttt{mirco.ciallella@u-pariscite.fr} \\
\And
\href{https://orcid.org/0000-0001-6698-2623}{\includegraphics[scale=0.06]{Figures/orcid.pdf}\hspace{1mm}Jens Visbech}$^*$ \\
	Department of Applied Mathematics and Computer Science\\
	Technical University of Denmark\\
	2800, Kgs. Lyngby \\
	\texttt{jvis@dtu.dk} \\
}

\renewcommand{\headeright}{Ciallella and Visbech (2026)}
\renewcommand{\shorttitle}{An \textit{arXiv} Preprint}

\hypersetup{
pdftitle={},
pdfsubject={},
pdfauthor={},
pdfkeywords={},
}

\begin{document}

\maketitle

$^*$ Shared first-authorship; both authors contributed equally to the completion of this work. \\
$^\dagger$ Corresponding author.

\vspace{5mm}

\begin{abstract}
	Higher-order finite element methods are effective for solving partial differential equations. However, applying them in complex curved domains is often difficult due to challenges in creating high-quality curvilinear meshes. Unfitted, or embedded, methods provide a valid alternative by avoiding complete mesh generation and curved element integration, but their accuracy can suffer from geometric errors introduced during embedding. In this work, we explore a new family of polynomial corrections obtained by solving a local minimization problem to improve the consistency of embedded boundary methods. This approach generalizes existing techniques such as the shifted boundary method (SBM) and the reconstruction for off-site data (ROD) method. Unlike the SBM, which depends on truncated Taylor expansions, the proposed family of polynomial corrections is derived from a constrained minimization problem, similar to the ROD method. We demonstrate that this approach yields better system conditioning than the original SBM and provides a flexible framework that can be applied pointwise, without solving the full ROD linear system for each boundary element. This paper presents four formulations based on different functionals and extends the family of minimization-based polynomial corrections to handle Neumann and Robin boundary conditions, demonstrating that an elegant formulation is possible in this context. Several numerical experiments for the Poisson problem are presented to show that the generalized polynomial corrections achieve high-order accuracy across all boundary conditions.
\end{abstract}


\section{Introduction}\label{sec:intro}

Many real-world problems in computational physics are governed by partial differential equations (PDEs) defined on complex domains with curved geometries. Examples can be found across various engineering and physical applications, including the simulation of blood flow in patient-specific vascular networks \cite{taylor2009patient}, the prediction of ice build-up on aircraft surfaces \cite{beaugendre2003fensap}, and shape optimization problems in structural and aerodynamic design \cite{imam1982three}. Other instances include noise propagation around aircrafts \cite{tam2004computational}, room acoustics \cite{pind2019time}, and free-surface flows \cite{scardovelli1999direct}, hereunder nonlinear water wave simulations \cite{melander2025high, engsigkarup2016stabilised, visbech2025fnpf}. Numerical methods, such as the finite element method (FEM) \cite{zienkiewicz1977finite}, are widely used to solve these problems. In particular, high-order methods, such as the spectral element method (SEM) \cite{karniadakis2005spectral}, can achieve very high accuracy even on relatively coarse meshes. If the solution is smooth, the discretization errors will decay exponentially as the polynomial basis order increases. However, implementing these on curved domains presents significant challenges, primarily due to the need for precise geometric representation and the complexity of numerical integration over curved elements. To avoid the difficulties of generating high-quality body-fitted meshes, unfitted (also referred to as embedded or immersed) meshes can be a crucial alternative.
The original concept behind these methods traces back to the volume penalization approach introduced by Peskin in the 1970s \cite{peskin1977numerical,peskin2002immersed}. Since then, many techniques have been proposed in both finite difference \cite{leveque1994immersed,mittal2005immersed}, finite volume \cite{cottet2008eulerian,ciallella2025semi}, and finite element \cite{boffi2003finite,boffi2015finite} frameworks. In contrast to volume penalization, the finite element context has seen the development of approaches that replace the exact boundary with a less precise representation, starting from the generalized finite element method (GFEM) \cite{strouboulis2000design}, the extended finite element method (XFEM) \cite{hansbo2002unfitted}, and the cut finite element method (CutFEM) \cite{burman2015cutfem}. However, dealing with cut elements can result in geometrically ill-posed configurations and poor conditioning. To address these challenges, numerical methods for cut elements with enhanced stability properties have been successfully developed in \cite{burman2010ghost, burman2015cutfem, burman2025cut, burman2014fictitious}. Recently, the idea of creating embedded methods that completely avoid handling cut elements has emerged as a turning point, eliminating the need for additional stabilization and bypassing the geometric complexity of defining cut elements through intersections and similar operations. In this approach, the background mesh is independent of the domain boundaries, and the computational domain is formed simply by removing elements outside the physical domain. To prevent the small-cell problem, the elements intersected by the physical geometry are also discarded, and boundary conditions are imposed on the so-called surrogate boundary instead of on the cut-cell boundary. However, this geometric approximation is only first-order accurate, meaning the discretization error is proportional to the mesh size, $h$.\\
This is because the total discretization error combines the error inherent in the numerical method for the PDE with the error arising from the geometrical approximation of the boundary. Embedded boundary methods, such as the shifted boundary method (SBM) \cite{main2018shifted,song2018shifted}, compensate for this geometrical error by appropriately adjusting the boundary conditions on the computational domain. The recently developed SBM imposes boundary conditions consistently through truncated Taylor expansions and has been successfully extended to problems with moving interfaces \cite{atallah2025high, colomes2026generalized, ciallella2022extrapolated,carlier2023enriched} as well as to high-order accuracy \cite{atallah2022high, ciallella2023shifted, collins2023penalty}. In \cite{ciallella2023shifted}, the authors developed an efficient way to achieve an arbitrary high-order SBM on curved domains, but using affine elements. This was done without explicitly computing the higher-order derivatives of the Taylor expansion. The method, called the shifted boundary polynomial correction, is based on off-element evaluation (extrapolation) of the local basis functions onto the true physical boundary. The polynomial correction-based SBM was then extended to Neumann and Robin boundaries on completely unfitted meshes in \cite{visbech2025spectral} in the context of the Poisson problem. Here, the authors addressed issues related to SBM conditioning at very high polynomial orders with the classical extrapolation of the local basis functions. By changing the element selection and the mapping between the true and computational boundaries, significant differences in conditioning behavior were observed. The polynomial corrections have later been applied in various works, e.g., \cite{collins2026gap, BOSCHERI2025114215, visbech2026recent}
In parallel, another research group developed an alternative embedded approach called the reconstruction for off-site data (ROD) method \cite{costa2018very,costa2019very}, to achieve high-order accuracy on curved boundaries discretized with simplex elements within a finite volume framework. The concept behind ROD is to create a modified polynomial in the boundary element that closely matches the internal one, while precisely enforcing the boundary condition on the actual boundary using Lagrange multipliers. This minimization-based approach results in a linear system that must be solved locally at each boundary element. The method was subsequently extended to the discontinuous finite element framework for elliptic and hyperbolic PDEs \cite{santos2024very,ciallella2024very}. Recent analysis in the context of Dirichlet boundary conditions for hyperbolic PDEs \cite{ciallella2025minimization} showed that the solution of the ROD minimization problem can be reformulated as a simple polynomial correction with a different scaling coefficient than in the SBM case. Notably, although this polynomial correction reduces to a reformulation of the ROD approach in one dimension, in multiple space dimensions, it can be applied point-wise in the spirit of the SBM, resulting in a method that maintains the minimization-based nature while avoiding the inversion of any linear system. This scaling coefficient has been shown to be essential for the method’s stability with hyperbolic PDEs. Indeed, whereas the SBM exhibits strict stability restrictions \cite{ciallella2025stability}, ROD-type methods are found to be more robust \cite{ciallella2025minimization}.\\

\subsection{Paper contributions}

In this paper, we investigate the use of minimization-based polynomial corrections in the context of elliptic PDEs, i.e., a Poisson problem with an reaction term. We demonstrate that, without resorting to any special element selection or mapping \cite{visbech2025spectral}, the system's conditioning can be improved over the original SBM. Furthermore, we show that the original ROD minimization problem can be modified to yield a wide family of arbitrary high-order minimization-based polynomial corrections, each with potentially distinct properties. Since complex applications typically involve a mix of Dirichlet, Neumann, and Robin boundary conditions to model physical phenomena, we also extend the novel framework to obtain elegant polynomial corrections also for general Neumann and Robin boundary conditions on unfitted meshes. This novel framework is herein developed within the high-order spectral element method (SEM) due to \cite{patera1984spectral}, which combines a single-domain spectral collocation method with a multi-domain finite element approach. 

\subsection{Paper outline}

Following this section's introduction, we outline the main mathematical problem and preliminaries in Section \ref{sec:math_prop}, including details on how the unfitted domain is constructed and the spectral element approximation. Next, in Section \ref{sec:var_forms}, we present the Nitsche-based variational formulations, starting with a conformal boundary and then moving to an unfitted boundary. After that, we introduce the original shifted-boundary polynomial correction and the reconstruction-for-off-site-data methodology in Section \ref{sec:sbm_rod}. As a natural extension, Section \ref{sec:PolyCorrDir} highlights the novelty of this work: the family of minimization-based polynomial corrections that employ two different norms and two boundary condition constraints. Finally, numerical results for 1D and 2D cases are presented in Sections \ref{sec:results_1D} and \ref{sec:results_2D}, respectively, followed by a conclusion in Section \ref{sec:conclusion}.
\section{Mathematical Problem and Preliminaries}\label{sec:math_prop}

In a two-dimensional (2D) domain, $\Omega$, we consider the classical Poisson problem with a reaction term subject to Dirichlet, Neumann, and Robin boundary conditions. With that, the completely enclosing boundary is $\partial\Omega = \Gamma = \Gamma_D \cup \Gamma_N \cup \Gamma_R$. Our task is to find $u \in C^2(\Omega)$, such that
\begin{subequations} \label{eq:Poisson_strong}
\begin{align} 
        -\boldsymbol{\nabla}^2 u + \beta u  &= f, \quad \text{in} \quad \Omega,  \\
        u  &= u_D, \quad \text{on} \quad \Gamma_D, \\
        \boldsymbol{\nabla} u \cdot \boldsymbol{n}  &= q_N, \quad \text{on} \quad \Gamma_N, \\
         u + \varepsilon \boldsymbol{\nabla} u \cdot \boldsymbol{n}  &= u_{RD}+\varepsilon q_{RN} = g_R,  \quad \text{on} \quad \Gamma_R,
\end{align}
\end{subequations}
with $\boldsymbol{n} = (n_x,n_y)$ being the outward-facing normal on $\Gamma$, $\beta \geq 0$ is a scalar for the reaction term to ensure well-posedness if $\Gamma_D = \emptyset$. For this work, $\beta = 1$, thus omitted in the following. Moreover, $\boldsymbol{\nabla} = (\partial_x,\partial_y)$ is Cartesian differential operators and $f$ is a known forcing function. We also consider the one-dimensional (1D) domain for simplicity and to enable better visualization and proof-of-concept analysis. In this case, $\boldsymbol{n} = n = n_x$ and $\boldsymbol{\nabla} = \nabla = \partial_x$.

\subsection{Unfitted computational domain}\label{sec:unfitted_comp_domain}

In this work, we consider classical finite element tesselations of the \textit{true} domain, $\Omega$. However, in contrast to the more traditional boundary-fitted approaches, where the mesh-tessellation is performed so that elements conform to the domain and its boundaries, we consider an unfitted approach. Here, the mesh generation is decoupled from the domain geometry, and the unfitted computational \textit{surrogate} domain, denoted $\tilde{\Omega}_h$, can be constructed as follows: 
\begin{itemize}
    \item[i)] Generating a simple structured background mesh, $\mathcal{T}_h$, about which $\text{clos}(\Omega) \subseteq \mathcal{T}_h$, using only quadrilaterals for this work.
    \item[ii)] Then, embedding $\Omega$ into this mesh and identifying, elements inside the domain, $\mathcal{T}_h^{\text{in}}$, outside the domain, $\mathcal{T}_h^{\text{out}}$, and on the boundary, $\mathcal{T}_h^{\text{on}}$, such that $\mathcal{T}_h = \mathcal{T}_h^{\text{in}} \cup \mathcal{T}_h^{\text{out}} \cup \mathcal{T}_h^{\text{on}}$. This procedure can be performed in various ways; however, we base it on a level-set description of the true boundary. To support the understanding, see Figure \ref{fig:unfitted_domain_illustration} (top left).
    \item[iii)] The unfitted computational domain, $\tilde{\Omega}_h$, is constructed by selecting the appropriate elements to represent $\Omega$. In the context of the SBM, the traditional choice have been to set $\tilde{\Omega}_h = \mathcal{T}_h^{\text{in}}$ \cite{main2018shifted}, e.g., see Figure \ref{fig:unfitted_domain_illustration} (top-right), but other variant also exists, see e.g., \cite{visbech2025spectral,yang2024optimal,colomes2026generalized}, including some of or all of $\mathcal{T}_h^{\text{on}}$.
\end{itemize}

\begin{figure}
    \centering
    \includegraphics[scale=0.75]{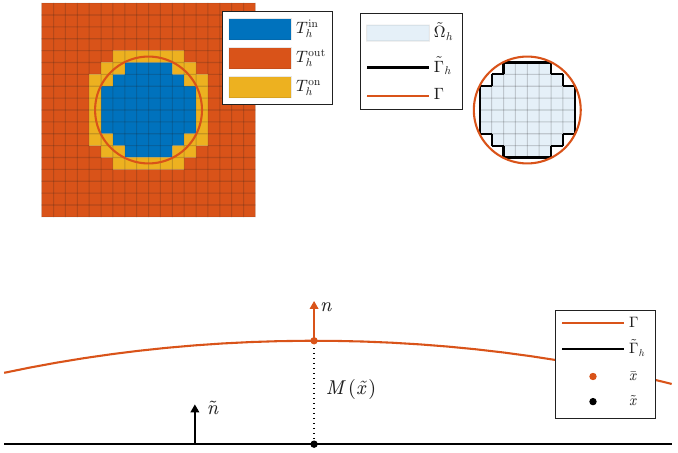}
    \caption{Top-left: The background mesh, $\mathcal{T}_h$, and its identifications based on level sets (in, out, on), $\mathcal{T}_h^{\text{in}}$, $\mathcal{T}_h^{\text{out}}$, and $\mathcal{T}_h^{\text{on}}$. Top-right: The true boundary, $\Gamma$, and the unfitted computational domain, $\tilde{\Omega}_h$, and its boundary, $\tilde{\Gamma}_h$. Bottom: Notation for normals and points on the true and computational boundary and mapping between them. }
    \label{fig:unfitted_domain_illustration}
\end{figure}

The final unfitted computational domain, $\tilde{\Omega}_h$, will consists of $N_{\text{elm}}$ non-overlapping affine elements as $ \tilde{\Omega}_h  = \bigcup_{n=1}^{N_{\text{elm}}} \tilde{\Omega}^n_h$ with $\tilde{\Omega}^n_h$ being the $n$'th element of $\tilde{\Omega}_h$. The boundary of $\tilde{\Omega}_h$ is denoted as $\tilde{\Gamma}_h = \tilde{\Gamma}_{h,D} \cup \tilde{\Gamma}_{h,N} \cup \tilde{\Gamma}_{h,R}$ with parts for imposing Dirichlet, Neumann, and Robin boundary conditions, respectively, if applicable, and simply constructed as the collection of element edges without an adjacent neighbor. Moreover, their associated normals are denoted as $\tilde{\boldsymbol{n}} = (\tilde{n}_x,\tilde{n}_y)$. Now that the concepts of the true and computational domains have been defined, another key ingredient is needed. Namely, a mapping, $\mathcal{M}$, between $\tilde{\Gamma}_h$ and $\Gamma$, where the boundary conditions are given on the latter. Using this mapping will enable us to define extensions of any function from $\tilde{\Gamma}_h$ to $\Gamma$ (and vice versa). The mapping is given as
\begin{equation}\label{eq:mapping}
     \mathcal{M}(\tilde{\boldsymbol{x}}): \tilde{\Gamma}_h \mapsto \Gamma, \quad \quad \quad \tilde{\boldsymbol{x}} \mapsto \bar{\boldsymbol{x}},
\end{equation}
where $\tilde{\boldsymbol{x}} = (\tilde{x},\tilde{y}) \in \tilde{\Gamma}_h$ and $\bar{\boldsymbol{x}} = (\bar{x},\bar{y}) \in \Gamma$. From this, we define a distance function as $\boldsymbol{d} = \boldsymbol{d}(\tilde{\boldsymbol{x}}) =\bar{\boldsymbol{x}} - \tilde{\boldsymbol{x}}$. For this work, we use the mapping defined along the true normal, $\boldsymbol{n}$, on $\Gamma$. Other choices are valid as well; see, e.g., \cite{visbech2025spectral}. See Figure \ref{fig:unfitted_domain_illustration} (bottom) for a conceptualization of the aforementioned quantities.

\subsection{Spectral element approximation}\label{sec:SEM}

Next, we introduce a high-order spectral element approximation to the computational domain. See \cite{karniadakis2005spectral,hesthaven2008nodal} and the references herein for element-based polynomial expansions. For this, we consider the following finite element space of piecewise continuous polynomials of degree $p$ as
\begin{equation}\label{eq:Wp}
\mathcal{W}_h^p = \left\{ w \in C^0(\Omega) : \forall n \in \{1,\ldots,N_{\text{elm}}\},\; w|_{\tilde{\Omega}^n_h} \in \mathbb{P}^p(\tilde{\Omega}^n_h) \right\}.
\end{equation}
As we solely consider structures of quadrilaterals, all geometrical and polynomial approximation quantities are constructed from tensor products of 1D spaces/geometries.

The global discrete solution, $u_h$, can be expressed as $N_{\text{elm}}$ local solutions, $u_h^n$, on each element, $\tilde{\Omega}^n_h$, as 
\begin{equation}
    u_h = \bigoplus_{n = 1}^{N_{\text{elm}}} u^n_h, \quad \text{with} \quad u^n_h \approx \sum_{m=1}^{N_{\text{ep}}} \hat{u}^n_{h,m} \psi_{m}=\sum_{i=1}^{N_{\text{ep}}} u^n_{h,i} \phi_{i},
\end{equation}
where each local solution equivalently is expressed using modal or nodal basis functions of up to degree $p$. For the modal approximation, the sum is over $N_{\text{ep}}$ modes (expansion coefficients), $\hat{u}^{n}_{h,m}$, and basis functions, $\psi_{m}$, taken as Legendre polynomials. For the nodal approximation, the sum is over $N_{\text{ep}}$ local solution values, $u^n_{h,i}$, and their associated Lagrange polynomials, $\phi_i$. The $N_{\text{ep}}$ quadrature element points is taken as Gauss-Lobatto-Legendre (GLL) distributed and connected to $p$ as $N_{\text{ep}} =(p+1)^2$ \cite{ronquist1987legendre}. In vectorized form, we have
\begin{equation}
    u_h^n = \boldsymbol{\psi}^T(x) \hat{\boldsymbol{u}}^n_h = \boldsymbol{\phi}^T(x) \boldsymbol{u}^n_h,
\end{equation}
where 
$\boldsymbol{\psi} = [\psi_1,\ldots,\psi_{N_{\text{ep}}}]^T$,
$\hat{\boldsymbol{u}}^n_h = [\hat{u}^n_{h,1},\ldots,\hat{u}^n_{h,N_{\text{ep}}}]^T$,
$\boldsymbol{\phi} = [\phi_1,\ldots,\phi_{N_{\text{ep}}}]^T$, and
$\boldsymbol{u}^n_h = [u^n_{h,1},\ldots,u^n_{h,N_{\text{ep}}}]^T$. The nodal-modal duality is given through a generalized Vandermonde matrix as
\begin{equation}
    \boldsymbol{u}^n_h = \boldsymbol{\mathcal{V}} \hat{\boldsymbol{u}}^n_h, \quad \text{where} \quad \boldsymbol{\mathcal{V}}_{ij} = \psi_j(\boldsymbol{r}_i),  
\end{equation}
where $\boldsymbol{r}$ is obtained through an inverse affine mapping to a reference domain, $\mathcal{R}$, as this work is entirely iso-parametric. The affine mapping from any physical elements, $\tilde{\Omega}^n_h$, to this reference space is $\Psi = \Psi(\boldsymbol{r}): \tilde{\Omega}^n_h \mapsto \mathcal{R}$, such that $\boldsymbol{x} = \Psi(\boldsymbol{r})$ and $\boldsymbol{r} = \Psi^{-1}(\boldsymbol{x})$, where $\boldsymbol{r} = (r,s)$ and $\boldsymbol{x} = (x,y)$. We have $\mathcal{R} = \{\boldsymbol{r} = (r,s) | (r, s)  \geq -1 ; (r,s) \leq 1 \}$. Following \cite{hesthaven2008nodal}, all nodal spectral element operators, e.g., for integration and differentiation, are built using modal elemental contributions, which are then converted into nodal space using these generalized Vandermonde matrices. Consult \cite{hesthaven2008nodal} for in-depth details. 
\section{Variational formulations: From conformal to unfitted domains}\label{sec:var_forms}

For completeness, consider an exact tesselation of $\Omega$ that is completely conformal to its boundaries and denote its mesh by $\Omega_h$ with boundaries $\Gamma_{h,D}$, $\Gamma_{h,N}$, and $\Gamma_{h,R}$. Assume that our finite elements space, $\mathcal{W}_h^p$, from \eqref{eq:Wp} is applicable to $\Omega_h$ as well. Then, we recast the Poisson problem in \eqref{eq:Poisson_strong} in a generalized weak form based on the principles of Nitsche's method, \cite{Nitsche1971,juntunen2009nitsche,benzaken2024constructing,visbech2025spectral}. We encourage the reader to consult \cite{visbech2025spectral} for more details on the asymptotic-preserving properties, among others. The generalized variational problem is to find $u_h \in \mathcal{W}^p_h$ such that $\forall w_h \in\mathcal{W}^p_h$, where
\begin{equation}
    a \langle u_h,w_h \rangle = b\langle w_h \rangle,
\end{equation}
with the following decomposition, based on domain and various boundary contributions, of the bilinear and linear form
\begin{subequations}
\begin{align}
     a \langle u_h,w_h \rangle &= a\langle u_h,w_h \rangle_{\Omega_h} + a\langle u_h,w_h \rangle_{\Gamma_{h,D}} + a \langle u_h,w_h \rangle_{\Gamma_{h,N}} + a \langle u_h,w_h \rangle_{\Gamma_{h,R}}, \\
     b\langle w_h \rangle &= b\langle w_h \rangle_{\Omega_h} + b\langle w_h \rangle_{\Gamma_{h,D}} + b \langle w_h \rangle_{\Gamma_{h,N}} + b\langle w_h \rangle_{\Gamma_{h,R}},
\end{align}
\end{subequations}
where the individual terms are given as
\begin{subequations}
\begin{align}
    a\langle u_h,w_h \rangle_{\Omega_h} &= (\boldsymbol{\nabla} u_h, \boldsymbol{\nabla} w_h)_{\Omega_h} + \beta  ( u_h,  w_h)_{\Omega_h},  \\ 
    a\langle u_h,w_h \rangle_{\Gamma_{h,D}} &= \gamma^{-1} (u_h,w_h)_{\Gamma_{h,D}} - (\boldsymbol{\nabla} u_h \cdot \boldsymbol{n},w_h)_{\Gamma_{h,D}} -(u_h,\boldsymbol{\nabla} w_h \cdot \boldsymbol{n})_{\Gamma_{h,D}},  \\
    a \langle u_h,w_h \rangle_{\Gamma_{h,N}} &= 0, \\
    a \langle u_h,w_h \rangle_{\Gamma_{h,R}} &= \frac{1}{\varepsilon + \gamma} (u_h,w_h)_{\Gamma_{h,R}} 
    - \frac{\gamma}{\varepsilon + \gamma} (\boldsymbol{\nabla} u_h \cdot \boldsymbol{n},w_h)_{\Gamma_{h,R}} \notag \\
    & - \frac{\gamma}{\varepsilon + \gamma} (u_h,\boldsymbol{\nabla} w_h \cdot \boldsymbol{n})_{\Gamma_{h,R}}
    - \frac{\varepsilon\gamma}{\varepsilon + \gamma} (\boldsymbol{\nabla} u_h \cdot \boldsymbol{n},\boldsymbol{\nabla} w_h \cdot \boldsymbol{n})_{\Gamma_{h,R}} ,   
\end{align}
\end{subequations}
and
\begin{subequations}
\begin{align} 
    b\langle w_h \rangle_{\Omega_h} &= (f,w_h)_{\Omega_h},  \\
    b\langle w_h \rangle_{\Gamma_{h,D}} &= \gamma^{-1} (u_D,w_h)_{\Gamma_{h,D}} 
    -(u_D,\boldsymbol{\nabla} w_h \cdot \boldsymbol{n})_{\Gamma_{h,D}} ,  \\
    b \langle w_h \rangle_{\Gamma_{h,N}} &= (q_N,w_h)_{\Gamma_{h,N}},  \\
    b\langle w_h \rangle_{\Gamma_{h,R}} &=  \frac{1}{\varepsilon + \gamma} (g_{R},w_h)_{\Gamma_{h,R}} - \frac{\gamma}{\varepsilon + \gamma} (g_{R},\boldsymbol{\nabla} w_h \cdot \boldsymbol{n})_{\Gamma_{h,R}}. 
\end{align}
\end{subequations}
Here $(u,w)_{s}  =\int_{s} u\,w\, ds$ is denoting the classical $L^2$ scalar product over the set $s$. Moreover, $\gamma^{-1}$ is a penalty parameter governing the stabilization terms; see, e.g., \cite{juntunen2009nitsche, benzaken2024constructing} and the references herein. For this work, we set $\gamma = h/2$, with $h$ being the element length.

\subsection{Switching to the unfitted domain}
Now, consider the unfitted computational domain, $\tilde{\Omega}_h$ and its boundary, $\tilde{\Gamma}_h = \tilde{\Gamma}_{h,D} \cup \tilde{\Gamma}_{h,N} \cup \tilde{\Gamma}_{h,R}$. Without having given any thought to the boundary conditions, we seek $u_h \in \mathcal{W}^P_h$ such that $\forall w_h \in\mathcal{W}^P_h$, where
\begin{equation}\label{eq:generalized_form}
    a \langle u_h,w_h \rangle = b\langle w_h \rangle,
\end{equation}
with
\begin{subequations}
\begin{align}
     a \langle u_h,w_h \rangle &= a\langle u_h,w_h \rangle_{\tilde{\Omega}_h} + a\langle u_h,w_h \rangle_{\tilde{\Gamma}_{h,D}} + a \langle u_h,w_h \rangle_{\tilde{\Gamma}_{h,N}} + a \langle u_h,w_h \rangle_{\tilde{\Gamma}_{h,R}}, \\
     b\langle w_h \rangle &= b\langle w_h \rangle_{\tilde{\Omega}_h} + b\langle w_h \rangle_{\tilde{\Gamma}_{h,D}} + b \langle w_h \rangle_{\tilde{\Gamma}_{h,N}} + b\langle w_h \rangle_{\tilde{\Gamma}_{h,R}},
\end{align}
\end{subequations}
where
\begin{subequations}\label{eq:unfitted_bilinear_form}
\begin{align}
    a\langle u_h,w_h \rangle_{\tilde{\Omega}_h} &= (\boldsymbol{\nabla} u_h, \boldsymbol{\nabla} w_h)_{\tilde{\Omega}_h} + \beta  ( u_h,  w_h)_{\tilde{\Omega}_h},  \\
    a\langle u_h,w_h \rangle_{\tilde{\Gamma}_{h,D}} &= \gamma^{-1} (u_h,w_h)_{\tilde{\Gamma}_{h,D}} - (\boldsymbol{\nabla} u_h \cdot \tilde{\boldsymbol{n}},w_h)_{\tilde{\Gamma}_{h,D}} -(u_h,\boldsymbol{\nabla} w_h \cdot \tilde{\boldsymbol{n}})_{\tilde{\Gamma}_{h,D}},  \\
    a \langle u_h,w_h \rangle_{\tilde{\Gamma}_{h,N}} &= 0, \\
    a \langle u_h,w_h \rangle_{\tilde{\Gamma}_{h,R}} &= \frac{1}{\varepsilon + \gamma} (u_h,w_h)_{\tilde{\Gamma}_{h,R}} 
    - \frac{\gamma}{\varepsilon + \gamma} (\boldsymbol{\nabla} u_h \cdot \tilde{\boldsymbol{n}},w_h)_{\tilde{\Gamma}_{h,R}} \notag \\
    & - \frac{\gamma}{\varepsilon + \gamma} (u_h,\boldsymbol{\nabla} w_h \cdot \tilde{\boldsymbol{n}})_{\tilde{\Gamma}_{h,R}}
    - \frac{\varepsilon\gamma}{\varepsilon + \gamma} (\boldsymbol{\nabla} u_h \cdot \tilde{\boldsymbol{n}},\boldsymbol{\nabla} w_h \cdot \tilde{\boldsymbol{n}})_{\tilde{\Gamma}_{h,R}} ,   
\end{align}
\end{subequations}
and
\begin{subequations}\label{eq:unfitted_linear_form}
\begin{align} 
    b\langle w_h \rangle_{\tilde{\Omega}_h} &= (f,w_h)_{\tilde{\Omega}_h},  \\
    b\langle w_h \rangle_{\tilde{\Gamma}_{h,D}} &= \gamma^{-1} (\tilde{u}_D,w_h)_{\tilde{\Gamma}_{h,D}} 
    -(\tilde{u}_D,\boldsymbol{\nabla} w_h \cdot \tilde{\boldsymbol{n}})_{\tilde{\Gamma}_{h,D}} ,  \\
    b \langle w_h \rangle_{\tilde{\Gamma}_{h,N}} &= (\tilde{q}_N,w_h)_{\tilde{\Gamma}_{h,N}},  \\
    b\langle w_h \rangle_{\tilde{\Gamma}_{h,R}} &=  \frac{1}{\varepsilon + \gamma} (\tilde{g}_{R},w_h)_{\tilde{\Gamma}_{h,R}} - \frac{\gamma}{\varepsilon + \gamma} (\tilde{g}_{R},\boldsymbol{\nabla} w_h \cdot \tilde{\boldsymbol{n}})_{\tilde{\Gamma}_{h,R}},
\end{align}
\end{subequations}
where the unknown Dirichlet, Neumann, and Robin boundary condition values are defined as
\begin{subequations}\label{eq:unknown_boundary_data}
\begin{align}
    \tilde{u}_D    &= u(\tilde{\boldsymbol{x}}),                                          & \text{for} \quad  \tilde{\boldsymbol{x}} \in \tilde{\Gamma}_{h,D},  \\
    \tilde{q}_N    &= \boldsymbol{\nabla} u(\tilde{\boldsymbol{x}}) \cdot \tilde{\boldsymbol{n}},   &  \text{for} \quad  \tilde{\boldsymbol{x}} \in \tilde{\Gamma}_{h,N},  \\
    \tilde{g}_{R} &= u(\tilde{\boldsymbol{x}}) + \varepsilon \boldsymbol{\nabla} u(\tilde{\boldsymbol{x}}) \cdot \tilde{\boldsymbol{n}},                                          & \text{for} \quad  \tilde{\boldsymbol{x}} \in \tilde{\Gamma}_{h,R}. 
\end{align}
\end{subequations}

The main challenge now is how to represent $\tilde{u}_D$ in terms of $u_D$ and related quantities for Neumann and Robin boundary conditions. Simply equating those will result in a first-order-accurate scheme; something more elaborate is needed, as will be covered in the following.
\section{The shifted boundary polynomial correction and reconstruction for off-site data}\label{sec:sbm_rod}

Next, we introduce a new framework based on the so-called {\it polynomial correction} to develop general embedded boundary conditions for unfitted domains. The original idea \cite{ciallella2023shifted} was introduced for Dirichlet boundary conditions in problems on curved domains with affine elements (as a curvature approximation/correction). It simplifies the implementation of the shifted boundary method (SBM) \cite{main2018shifted,atallah2022high} and avoids the need to compute higher-order derivatives via truncated Taylor expansions. The polynomial correction was extended to general Neumann and Robin conditions on fully unstructured meshes in \cite{visbech2025spectral} for elliptic problems. Meanwhile, a different family of embedded boundary methods, called the reconstruction for off-site data (ROD) method \cite{costa2018very,costa2019very}, was developed. Unlike SBM, which uses Taylor expansions, ROD solves a minimization problem to compute modified high-order boundary conditions on the computational boundary. The ROD method allows for the direct incorporation of general Neumann and Robin boundary conditions by defining a constraint matrix. In this section, we present the two families of methods and compare their differences, before introducing a completely new, generalized, minimization-based framework for deriving polynomial corrections in Section \ref{sec:PolyCorrDir}.

\subsection{The Shifted boundary polynomial correction}\label{sec:SBcorrection}
This section briefly reviews the original polynomial correction for the SBM \cite{ciallella2023shifted} to establish notation and provide context for the new formulations derived in Section \ref{sec:PolyCorrDir}. The SBM, especially the high-order variant \cite{atallah2022high}, enforces boundary conditions at the computational boundary $\tilde{\Gamma}_h$ by applying a corrected Dirichlet value obtained via a truncated Taylor expansion at an appropriate order. Considering the pure Dirichlet case, $\Gamma = \Gamma_D$ and $\tilde{\Gamma}_h = \tilde{\Gamma}_{h,D}$, in 1D. Recall that $u_D$ is the essential boundary condition at $\bar{x} \in \Gamma$, which does not coincide with $\tilde x \in \tilde{\Gamma}_h$. A corrected boundary condition $\tilde{u}_D = \tilde{u}_D(\tilde x)$ for the computational surrogate/unfitted boundary can be constructed starting from the truncated $p$'th order Taylor expansion given as
\begin{equation}\label{eq:taylor}
    u(\bar x)=u(\tilde x + d) \approx u(\tilde x) + \sum_{m=1}^p \partial_x^{(m)} u(\tilde x) \frac{d^m}{m!}.
\end{equation}
With this, the high-order SBM \cite{atallah2022high} can be developed, including all high-order derivative terms, and by enforcing the prescribed boundary condition $u(\bar x) = u_D$ as
\begin{equation}
    \tilde{u}_D = \tilde{u}_D(\tilde x) = u_D - \sum_{m=1}^p \partial_x^{(m)} u_h(\tilde x) \frac{d^m}{m!},  
\end{equation}
with $u_h(x) = \boldsymbol{\phi}^T(x)\boldsymbol{u}$ being the internal solution of the boundary element, where $\boldsymbol{u}$ are the solution values and $\boldsymbol{\phi}$ is the vector of nodal basis functions as given in Section \ref{sec:SEM}.

As shown in \cite{ciallella2023shifted}, from \eqref{eq:taylor} it is direct to infer that the sum of derivative terms in the expression is equivalent to the difference of the polynomial approximation evaluated at the two boundary locations. Therefore, the shifted boundary Dirichlet boundary condition, $\tilde u_D$ on $\tilde{\Gamma}_h$ can be rewritten using the following correction
\begin{equation}\label{eq:SBcorrection}
\begin{split}
    \tilde{u}_D   &= u_h(\tilde x) - \alpha_D^{\text{SBM}}u_h(\bar x) + \alpha_D^{\text{SBM}}u_D \\
                  &= [\boldsymbol{\phi}(\tilde x) - \alpha_D^{\text{SBM}}\boldsymbol{\phi}(\bar x)]^T\, \boldsymbol{u} + \alpha_D^{\text{SBM}}u_D,
\end{split}
\end{equation}
where $\alpha_D^{\text{SBM}} = 1$ and the use of this notation will be evident in the following parts.

It should be noticed that \eqref{eq:SBcorrection} is arbitrarily high-order accurate depending on the set of chosen basis functions. Moreover, it holds in multiple dimensions, since the Dirichlet condition imposed at a quadrature point is not influenced by the boundary condition at the next point. Therefore, in multiple dimensions, here 2D, the correction \eqref{eq:SBcorrection} simply becomes
\begin{equation}\label{eq:SBcorrection2D}
\tilde{u}_D = \tilde{u}_D(\tilde{\boldsymbol{x}}) = [\boldsymbol{\phi}(\tilde{\boldsymbol{x}})-\alpha_D^{\text{SBM}}\boldsymbol{\phi}(\bar{\boldsymbol{x}})]^T\, \boldsymbol{u} + \alpha_D^{\text{SBM}}u_D, 
\end{equation}
where $\boldsymbol{x} = (x, y)$ and $\boldsymbol{\phi}$ are the two-dimensional high-order nodal basis functions from Section \ref{sec:SEM}.

Switching the focus to the Neumann boundary condition, such that $\Gamma = \Gamma_N$ and $\tilde{\Gamma}_h = \tilde{\Gamma}_{h,N}$. To develop the shifted boundary polynomial correction in the Neumann case, we consider the reasoning done in \cite{visbech2025spectral}. Recall that $\boldsymbol{n}$ is the outward-facing normal on the true boundary, $\Gamma_N$, with an associated tangent, $\boldsymbol{t}=(t_x,t_y)$, whereas $\tilde{\boldsymbol{n}}$ is the outward-facing normal on the unfitted surrogate computational boundary, $\tilde{\Gamma}_{h,N}$. Following the same notation as above, we define $q_N$ as the exact Neumann boundary condition on the physical boundary and $\tilde q_N$ as the modified condition on the computational one. To transfer the condition imposed, we introduce the following normal decomposition
\begin{equation}\label{eq:normaldec}
    \tilde{\boldsymbol{n}} = (\tilde{\boldsymbol{n}} \cdot \boldsymbol{n}) \boldsymbol{n} + (\tilde{\boldsymbol{n}}\cdot \boldsymbol{t}) \boldsymbol{t}.
\end{equation}
A modified Neumann boundary condition on the computational boundary is obtained by using the decomposition on $\tilde q_N(\tilde{\boldsymbol{x}}) =   \tilde{\boldsymbol{n}} \cdot\boldsymbol{\nabla} u(\tilde{\boldsymbol{x}}) $ yielding
\begin{equation}\label{eq:neumanndec}
\tilde q_N = \tilde q_N (\tilde{\boldsymbol{x}}) = (\tilde{\boldsymbol{n}}\cdot \boldsymbol{n}) \boldsymbol{n}\cdot \nabla u(\tilde{\boldsymbol{x}}) + (\tilde{\boldsymbol{n}}\cdot \boldsymbol{t}) \boldsymbol{t}\cdot \nabla u_h(\tilde{\boldsymbol{x}}),
\end{equation}
where the first term, $\boldsymbol{n}\cdot \nabla u(\tilde{\boldsymbol{x}})$, is modified by using the polynomial correction applied to the derivatives as
\begin{equation}\label{eq:SBDerPolyCorr}
    \boldsymbol{n}\cdot \nabla u(\tilde{\boldsymbol{x}}) =  \boldsymbol{n}\cdot \nabla u_h(\tilde{\boldsymbol{x}}) - \boldsymbol{n}\cdot \nabla u_h(\bar{\boldsymbol{x}})  + q_N.
\end{equation}
Assembling \eqref{eq:neumanndec} and \eqref{eq:SBDerPolyCorr} and after some simple manipulations, the condition can be written as a polynomial correction based on the normal derivatives as
\begin{equation}\label{eq:SBNeumannPolyCorr} 
    \begin{split}
\tilde q_N = \tilde q_N (\tilde{\boldsymbol{x}}) &= \tilde{\boldsymbol{n}}\cdot \nabla u_h(\tilde{\boldsymbol{x}}) - \alpha_N^{\text{SBM}}\boldsymbol{n}\cdot \nabla u_h(\bar{\boldsymbol{x}}) + \alpha_N^{\text{SBM}} q_N\\
&=[\tilde{\boldsymbol{n}}\cdot \nabla \boldsymbol{\phi}(\tilde{\boldsymbol{x}}) - \alpha_N^{\text{SBM}}\boldsymbol{n}\cdot \nabla \boldsymbol{\phi}(\bar{\boldsymbol{x}})]^T \boldsymbol{u} + \alpha_N^{\text{SBM}} q_N,
    \end{split}
\end{equation} 
where $\alpha_N^{\text{SBM}} = \tilde{\boldsymbol{n}}\cdot\boldsymbol{n}$. This polynomial correction has the same shape as the Dirichlet one \eqref{eq:SBcorrection2D}, but rather than focusing on the basis functions, it focuses on the normal derivatives, and it has a multiplying scalar applied to the true boundary condition and to the evaluation of the basis functions in the true boundary.

\subsection{Reconstruction for off-site data}\label{sec:ROD}

Before presenting the new family of methods inspired by the ROD approach, we recall its original formulation in the context of finite element methods. In multiple dimensions, the standard ROD method as implemented in \cite{santos2024very,ciallella2024very} works as follows. Within each boundary element, a minimization problem is solved to obtain a single modified polynomial $\tilde u_D(\boldsymbol{x})=\boldsymbol{\phi}^T(\boldsymbol{x})\boldsymbol{v}$, which is as close as possible to the internal one $u_h(\boldsymbol{x})=\boldsymbol{\phi}^T(\boldsymbol{x})\boldsymbol{u}$ from Section \ref{sec:SEM}. In general, the minimization introduces as many constraints as the number of quadrature points used for the boundary integral on that element. These constraints enforce the correct boundary conditions at the corresponding points on the true boundary. Moreover, as shown in \cite{santos2024very}, increasing the number of constraints when increasing the polynomial degree is crucial to maintain the consistency of the scheme, equivalent to increasing the order of the truncated Taylor expansion in the context of the high-order SBM \cite{atallah2022high}. The constrained minimization problem is easily formulated by minimizing the distance between the polynomials using the Euclidean norm and enforcing the boundary conditions via Lagrange multipliers. When considering standard Dirichlet conditions, the functional reads
\begin{equation}\label{eq:rod2dfunctional}
\mathcal{L}(\boldsymbol{v}, \boldsymbol{\lambda}) = \frac12 \|\boldsymbol{v}-\boldsymbol{u}\|^2_2 + \sum_{k=1}^K \lambda_k \left( \tilde u_D(\bar{\boldsymbol{x}}_k) - u_D(\bar{\boldsymbol{x}}_k) \right),
\end{equation}
where $\boldsymbol{\lambda} = [\lambda_1, \ldots, \lambda_K]^T$ is the vector of Lagrange multipliers with $K=1$ in 1D and $K=p+1$ in 2D. The optimality conditions give
\begin{equation}
\frac{\partial \mathcal{L}}{\partial \boldsymbol{v}} = \boldsymbol{v} - \boldsymbol{u} + \Phi(\bar{\boldsymbol{x}}) \boldsymbol{\lambda}  = 0, \qquad \frac{\partial \mathcal{L}}{\partial \boldsymbol{\lambda}} = \Phi^T(\bar{\boldsymbol{x}})\boldsymbol{v}  - \boldsymbol{u}_D = 0
\end{equation}
where $[\Phi(\bar{\boldsymbol{x}})]_{j k} = \phi_j(\bar{\boldsymbol{x}}_k)$, and $\Phi(\bar{\boldsymbol{x}})$ is assumed to have full column rank. Moreover, we have defined $\boldsymbol{u}_D=[u_D(\bar{\boldsymbol{x}}_1),\ldots,u_D(\bar{\boldsymbol{x}}_K)]$ the set of boundary conditions used for the minimization problem.
For a Dirichlet boundary condition, the ROD method boils down to finding the polynomial coefficients of $\tilde u_D$ coming from the solution of the following linear system as
\begin{equation}\label{eq:RODE2Dsystem}
\begin{bmatrix} I & \Phi(\bar{\boldsymbol{x}}) \\ \Phi^T(\bar{\boldsymbol{x}}) & 0 \end{bmatrix} \begin{bmatrix}\boldsymbol{v} \\ \boldsymbol{\lambda} \end{bmatrix} = \begin{bmatrix} \boldsymbol{u} \\ \boldsymbol{u}_D \end{bmatrix}.
\end{equation}
This system can be solved using the Schur complement by
\begin{equation}
    \boldsymbol{v} = \boldsymbol{u} - \Phi(\bar{\boldsymbol{x}}) \boldsymbol{\lambda}, \quad \text{with} \quad  \boldsymbol{\lambda} = (\Phi^T(\bar{\boldsymbol{x}})\Phi(\bar{\boldsymbol{x}}))^{-1}(\Phi^T(\bar{\boldsymbol{x}})\boldsymbol{u}  - \boldsymbol{u}_D).
\end{equation}
The coefficients $\boldsymbol{v}$ can be used to compute the ROD Dirichlet condition by evaluating $\tilde u_D(\tilde{\boldsymbol{x}}_k)$ for $k=1,\ldots,K$.

It is straightforward to extend this approach to Neumann conditions since the only modification to be done is in the definition of the linear constraint of \eqref{eq:rod2dfunctional}. In this case, we are building a polynomial $\tilde u_N(\boldsymbol{x})=\boldsymbol{\phi}^T(\boldsymbol{x})\boldsymbol{v}$ such that the following functional is minimized
\begin{equation}\label{eq:rod2dfunctionalNeumann}
\mathcal{L}(\boldsymbol{v}, \boldsymbol{\lambda}) = \frac12 \|\boldsymbol{v}-\boldsymbol{u}\|^2_2 + \sum_{k=1}^K \lambda_k \left( \boldsymbol{n}(\bar{\boldsymbol{x}}_k) \cdot \nabla \tilde u_N(\bar{\boldsymbol{x}}_k) - q_N(\bar{\boldsymbol{x}}_k) \right),
\end{equation}
and the matrix $\Phi$ would take the shape $ [\Phi(\bar{\boldsymbol{x}})]_{j k} = \boldsymbol{n}(\bar{\boldsymbol{x}}_k) \cdot \nabla \phi_j(\bar{\boldsymbol{x}}_k)$, where $\boldsymbol{n}(\bar{\boldsymbol{x}}_k)$ is the local normal vector to the true boundary.
The vector $\boldsymbol{u}_N$ in the right-hand side is then filled by the vector $\boldsymbol{q}_N=[q_N(\bar{\boldsymbol{x}}_1),\ldots,q_N(\bar{\boldsymbol{x}}_K)]$.
Once the coefficients $\boldsymbol{v}$ are computed, the ROD Neumann boundary condition is $\tilde q_N(\tilde{\boldsymbol{x}}_k)=\tilde{\boldsymbol{n}}\cdot\nabla \tilde u_N(\tilde{\boldsymbol{x}}_k)$ for $k=1,\ldots,K$. Being $\tilde{\boldsymbol{n}}$ the normal to the surrogate boundary, which has no curvature, its value will not change in the same boundary element. 

It can be observed that the standard ROD method always requires a matrix inversion for each boundary element. This is because, when constructing a single polynomial per element, if constraints are not increased, the method loses consistency. Here, we demonstrate how to overcome this issue and establish different minimization problems that can be reformulated as simple polynomial corrections. By doing so, the new minimization-based approach can be applied directly to each boundary point while maintaining high-order consistency and avoiding matrix inversions.

\section{Minimization-based polynomial corrections}\label{sec:PolyCorrDir}

In the following, we introduce a family of polynomial corrections designed to achieve arbitrary-order accuracy on multi-dimensional embedded domains. Like ROD, these corrections rely on solving a minimization problem, which avoids the costly inversion of the constraint matrix. Building on the 1D analysis of \cite{ciallella2025minimization}, we extend the minimization-based polynomial corrections to multiple dimensions. These corrections can be applied point-wise at each boundary quadrature point, as in \eqref{eq:SBcorrection2D}, while still maintaining consistency. The main idea is to construct the corrected polynomial at a computational boundary point using a local 1D reference system aligned with the chosen mapping in \eqref{eq:mapping} that connects that point to the corresponding true boundary point. When only a single constraint (the boundary condition) is enforced, the resulting polynomial from the minimization is not sufficient to guarantee high-order accuracy for general problems \cite{santos2024very}. However, this simplified system can be solved explicitly, resulting in a simple polynomial correction that is applied independently at each boundary point. Therefore, instead of creating one polynomial per element, we generate as many polynomials as there are boundary quadrature points. We first present this formulation for Dirichlet conditions in 1D, then demonstrate its direct extension to higher dimensions. The extension is straightforward because the correction acts locally along the mapping direction at each boundary point, effectively reducing the multidimensional problem to a set of independent 1D corrections. This novel and general framework enables the systematic derivation of new polynomial corrections based on different cost functions, providing a family of new multi-dimensional numerical boundary conditions for embedded domains. Finally, we show that the same minimization approach naturally extends to Neumann and Robin conditions, thereby generalizing polynomial correction methods to a wide range of boundary conditions.

\subsection{The Dirichlet case}

Although the hypotheses used for the development of this new set of methods are different, we show here that all of them can be written in the same polynomial correction. Following the same notation as above, the modified boundary condition, $\tilde u_D$, in the surrogate boundary point, $\tilde{\boldsymbol{x}}$, considering the true boundary condition, $u_D$, on the true boundary point, $\bar{\boldsymbol{x}}$, reads
\begin{equation}\label{eq:generalPC}
    \tilde u_D = [\boldsymbol{\phi}(\tilde{\boldsymbol{x}}) - \alpha_D \boldsymbol{\phi}(\bar{\boldsymbol{x}})]^T \boldsymbol{u} + \alpha_D u_D, 
\end{equation}
where $\alpha_D$ is a scalar coefficient and the only variable that varies across methods. When $\alpha_D = 0$, the surrogate boundary condition reduces to the unconstrained polynomial evaluation $\tilde u_D = \boldsymbol{\phi}(\tilde{\boldsymbol{x}})^T \boldsymbol{u}$, while $\alpha_D = 1$ enforces the shifted boundary polynomial corrections in \eqref{eq:SBcorrection2D}.

\subsubsection{ROD-E polynomial correction}

The concept of implementing minimization-based polynomial corrections for embedded computations originates from the analysis of Dirichlet boundary conditions in 1D presented in \cite{ciallella2025minimization}, which will be recalled here for completeness and then extended to higher dimensions and to general Neumann and Robin boundary conditions.

Considering a 1D computational domain, the ROD minimization problem becomes significantly simpler because 1D elements only have $K=1$ boundary constraints to impose. In this setup, the true boundary point is $\bar x$ and the computational boundary point is $\tilde x$. Therefore, applying the ROD approach here involves constructing the polynomial $\tilde u_D(x)=\boldsymbol{\phi}^T(x)\boldsymbol{v}$, by minimizing the Euclidean distance to the internal polynomial $u_h(x)=\boldsymbol{\phi}^T(x)\boldsymbol{u}$, subject to the boundary constraint $\tilde u_D(\bar x) = u_D$. In this case, the functional is given by
\begin{equation}
    \mathcal{L}(\boldsymbol{v} , \lambda) = \frac12\|\boldsymbol{v}-\boldsymbol{u}\|^2_2 + \lambda ( \tilde u_D(\bar x) - u_D ),
\end{equation}
and the corresponding optimality conditions are
\begin{equation}\label{eq:optcondRODE}
    \frac{\partial \mathcal{L}}{\partial \boldsymbol{v}} = \boldsymbol{v} - \boldsymbol{u} + \lambda  \boldsymbol{\phi}(\bar x) = \boldsymbol{0} , \qquad \frac{\partial \mathcal{L}}{\partial \lambda} = \boldsymbol{\phi}^T(\bar x) \boldsymbol{v} - u_D = 0.
\end{equation}
This system can be written in the following matrix form
\begin{equation}\label{eq:RODsystem}
\begin{bmatrix} I & \boldsymbol{\phi}(\bar x) \\ \boldsymbol{\phi}^T(\bar x) & 0 \end{bmatrix} \begin{bmatrix}\boldsymbol{v} \\ \lambda \end{bmatrix} = \begin{bmatrix} \boldsymbol{u} \\ u_D \end{bmatrix}.
\end{equation}
By solving this system, the polynomial coefficients $\boldsymbol{v}$ are obtained, and the modified condition $\tilde{u}_D(\tilde x)=\boldsymbol{\phi}^T(\tilde x)\boldsymbol{v}$ can be directly computed.

We call this approach the ROD-E polynomial correction because the minimization problem is formulated in terms of the Euclidean norm. Although ROD-E can be seen as a reformulation of ROD in 1D, the two methods differ in multiple dimensions. Therefore, we give them distinct names from the very beginning.

\begin{proposition}
Computing the modified boundary condition by inverting the linear system \eqref{eq:RODsystem} and evaluating the polynomial $\tilde u_D$ in $\tilde x$ is equivalent to the following polynomial correction
\begin{equation}\label{eq:rodE_final}
    \tilde{u}_D = [\boldsymbol{\phi}(\tilde x) - \alpha^{\text{ROD-E}}_D \boldsymbol{\phi}(\bar x)]^T \boldsymbol{u} + \alpha^{\text{ROD-E}}_D u_D , \quad \quad \text{with}
\end{equation}
\begin{equation}\label{eq:alpha_rod_e}
    \alpha^{\text{ROD-E}}_D = \frac{\boldsymbol{\phi}(\tilde{x})^T \boldsymbol{\phi}(\bar{x})}{\|\boldsymbol{\phi}(\bar{x})\|^2_2}.   
\end{equation}
\end{proposition}

\begin{proof}
The solution of this minimization problem can be found starting from the first optimality condition \eqref{eq:optcondRODE} as
\begin{equation}
   \boldsymbol{v} = \boldsymbol{u} - \lambda  \boldsymbol{\phi}(\bar x), \qquad \boldsymbol{\phi}^T(\bar x) (\boldsymbol{u} - \lambda  \boldsymbol{\phi}(\bar x)) - u_D = 0.
\end{equation}
The Lagrange multiplier $\lambda$ is obtained and the modified degrees of freedom $\boldsymbol{v}$ reads
\begin{equation}
\boldsymbol{v} = \boldsymbol{u} - \left( \frac{\boldsymbol{\phi}(\bar{x})^T \boldsymbol{u} - u_D}{\boldsymbol{\phi}(\bar{x})^T \boldsymbol{\phi}(\bar{x})} \right) \boldsymbol{\phi}(\bar{x}).
\end{equation}

In order for this approach to be a polynomial correction, it is necessary to consider the evaluation of the $\tilde u_D$ in the computational boundary point $\tilde x$, which gives
\begin{equation}
    \tilde u_D = \boldsymbol{\phi}(\tilde{x})^T \boldsymbol{v} = \boldsymbol{\phi}(\tilde{x})^T \boldsymbol{u} - \left( \frac{\boldsymbol{\phi}(\bar{x})^T \boldsymbol{u} - u_D}{\|\boldsymbol{\phi}(\bar{x})\|^2_2} \right) \boldsymbol{\phi}(\tilde{x})^T \boldsymbol{\phi}(\bar{x}).
\end{equation}
It is now easy to infer that the ROD-E modified condition can be recast as the polynomial correction \eqref{eq:rodE_final}.
\end{proof}

Although the minimization problem has been written under the condition of a 1D computational domain, in this section, we present the possibility of developing them in multiple space dimensions, in the spirit of the shifted boundary correction \eqref{eq:SBcorrection2D}. The key hypothesis is that the embedded boundary conditions are formulated in a 1D reference space, along the chosen mapping, $\mathcal{M}$, from \eqref{eq:mapping} that aligns the point on the computational boundary, $\tilde{\boldsymbol{x}}$, with the corresponding point on the true boundary, $\bar{\boldsymbol{x}}$. Therefore, in multiple dimensions, the ROD-E polynomial correction would simply read
\begin{equation}\label{eq:rodE_final2D}
    \tilde{u}_D = [\boldsymbol{\phi}(\tilde{\boldsymbol{x}}) - \alpha^{\text{ROD-E}}_D \boldsymbol{\phi}(\bar{\boldsymbol{x}})]^T \boldsymbol{u} + \alpha^{\text{ROD-E}}_D u_D , \quad \quad \text{with}
\end{equation}
\begin{equation}
    \alpha^{\text{ROD-E}}_D = \frac{\boldsymbol{\phi}(\tilde{\boldsymbol{x}})^T \boldsymbol{\phi}(\bar{\boldsymbol{x}})}{\|\boldsymbol{\phi}(\bar{\boldsymbol{x}})\|^2_2}.   
\end{equation}
To simplify the notation, we have defined the modified boundary condition at the surrogate boundary $\tilde{\boldsymbol{x}}$ as $\tilde{u}_D := \tilde{u}_D(\tilde{\boldsymbol{x}})$. When referring to the evaluation of the polynomial at any other point $\boldsymbol{x}$, we explicitly use the notation $\tilde{u}_D(\boldsymbol{x})$.

We emphasize that the SBM polynomial correction in \eqref{eq:SBcorrection2D} and the ROD-E one in \eqref{eq:rodE_final2D} differ only in the value of the scalar $\alpha_D$, which is equal to $1$ in the SBM case. Therefore, in the original ROD approach, a single modified polynomial is used for the entire boundary element to impose all boundary constraints. In contrast, what we call the ROD-E method here consists of applying the polynomial correction \eqref{eq:rodE_final2D} at each boundary point, which requires no inversion. Hence, the ROD-E polynomial correction effectively builds multiple modified polynomials within each boundary element, each one adapted to consistently enforce the boundary condition at a specific point on the true boundary. The generality of this approach makes it possible to develop different methods by simply changing the functional of the minimization problem. Moreover, we will show below that all the following methods can always be recast as a simple polynomial correction with a different scalar factor $\alpha_D$.

\subsubsection{ROD-\texorpdfstring{$L^2$}{L2} polynomial correction}

In \cite{ciallella2025minimization}, it was proposed to develop a new polynomial correction by changing the distance function in the minimization problem. This led to a method with different stability properties with respect to the ROD-E approach. The ROD-$L^2$ method reconstructs a modified polynomial $\tilde u_D(\boldsymbol{x})=\boldsymbol{\phi}^T (\boldsymbol{x})\boldsymbol{v}$ that minimizes the $L^2$ distance to the internal polynomial $u_h$, subject to the Dirichlet condition $u_D$ on the true boundary $\bar{\boldsymbol{x}}$. The new functional reads
\begin{equation}
    \mathcal{L}(\boldsymbol{v}, \lambda) = \frac{1}{2} \|\tilde u_D - u_h\|^2_{L^2(\tilde \Omega_h)} + \lambda(\tilde u_D(\bar{\boldsymbol{x}}) - u_D),
\end{equation}
where $\tilde \Omega_h$ represents the considered 2D boundary element. Substituting the polynomial expansions, the mass matrix appears in the first optimality condition, and the problem can be rewritten in the following matrix form
\begin{equation}
\frac{\partial \mathcal{L}}{\partial \boldsymbol{v}} = M(\boldsymbol{v} - \boldsymbol{u}) + \lambda \boldsymbol{\phi}(\bar{\boldsymbol{x}}) = \boldsymbol{0},  \qquad \frac{\partial \mathcal{L}}{\partial \lambda} = \boldsymbol{\phi}(\bar{\boldsymbol{x}})^T \boldsymbol{v} - u_D = 0 .
\end{equation}
Again, by solving the following linear system, 
\begin{equation}\label{eq:RODL2system}
\begin{bmatrix} M & \boldsymbol{\phi}(\bar{\boldsymbol{x}}) \\ \boldsymbol{\phi}^T(\bar{\boldsymbol{x}}) & 0 \end{bmatrix} \begin{bmatrix}\boldsymbol{v} \\ \lambda \end{bmatrix} = \begin{bmatrix} M \boldsymbol{u} \\ u_D \end{bmatrix}, 
\end{equation}
the coefficients $\boldsymbol{v}$ are obtained, and the evaluated polynomial $\tilde u_D(\tilde{\boldsymbol{x}})$ can be computed.

\begin{proposition}
Computing the ROD-$L^2$ modified Dirichlet boundary condition by inverting the linear system \eqref{eq:RODL2system} and evaluating the polynomial $\tilde u_D$ in $\tilde{\boldsymbol{x}}$ is equivalent to the following polynomial correction
\begin{equation}\label{eq:rodL2_final}
    \tilde u_D = [\boldsymbol{\phi}(\tilde{\boldsymbol{x}}) - \alpha_D^{\text{ROD-$L^2$}} \boldsymbol{\phi}(\bar{\boldsymbol{x}})]^T \boldsymbol{u} + \alpha_D^{\text{ROD-$L^2$}} u_D , \quad \quad \text{with}
\end{equation}
\begin{equation}\label{eq:alpha_rod_L2}
    \alpha_D^{\text{ROD-$L^2$}} = \frac{\boldsymbol{\phi}(\tilde{\boldsymbol{x}})^T M^{-1} \boldsymbol{\phi}(\bar{\boldsymbol{x}})}{\|\boldsymbol{\phi}(\bar{\boldsymbol{x}})\|^2_{M^{-1}}}.
\end{equation}

\end{proposition}

\begin{proof}

The optimality conditions provides following relations for $\boldsymbol{v}$ and $\lambda$
\[
\boldsymbol{v} = \boldsymbol{u} - \lambda M^{-1} \boldsymbol{\phi}(\bar{\boldsymbol{x}}),\qquad \lambda = \frac{\boldsymbol{\phi}(\bar{\boldsymbol{x}})^T \boldsymbol{u} - u_D(\bar{\boldsymbol{x}})}{\boldsymbol{\phi}(\bar{\boldsymbol{x}})^T M^{-1} \boldsymbol{\phi}(\bar{\boldsymbol{x}})}.
\]
As it was done above, in order to write the ROD-$L^2$ as a polynomial correction, it is necessary to evaluate $\tilde u_D$ at the computational boundary $\tilde{\boldsymbol{x}}$
\begin{equation}\label{eq:rodL2poly}
\begin{split}
    \tilde u_D = \boldsymbol{\phi}(\tilde{\boldsymbol{x}})^T \boldsymbol{v} = \boldsymbol{\phi}(\tilde{\boldsymbol{x}})^T \boldsymbol{u} - \left( \frac{\boldsymbol{\phi}(\bar{\boldsymbol{x}})^T \boldsymbol{u} - u_D(\bar{\boldsymbol{x}})}{\|\boldsymbol{\phi}(\bar{\boldsymbol{x}})\|^2_{M^{-1}}} \right) \boldsymbol{\phi}(\tilde{\boldsymbol{x}})^T M^{-1} \boldsymbol{\phi}(\bar{\boldsymbol{x}}).
\end{split}
\end{equation}
From equation \eqref{eq:rodL2poly}, it is direct to recast the ROD-$L^2$ method as the polynomial correction \eqref{eq:rodL2_final}. 

\end{proof}

\subsubsection{ROD-E-w polynomial correction}

In this section, we propose developing a variant of the ROD-E and ROD-$L^2$ methods, as opposed to the embedded approaches developed so far, which impose the boundary condition as a strict constraint. The main idea is to formulate a cost function that incorporates the boundary condition as a weak constraint instead of using a Lagrange multiplier. Although less common, we demonstrate here that this approach can also achieve arbitrarily high-order accuracy and can be represented as a polynomial correction. We note that in this work, we focus on elliptic PDEs solved with continuous Galerkin methods; however, these approaches could be directly applied to other computational frameworks and PDEs, potentially leading to different results. To illustrate a possible implementation, we consider the Euclidean norm for the polynomial coefficients. Consequently, in this case, the functional simply reads
\begin{equation}
     \mathcal{L}(\boldsymbol{v}) = \frac12 \|\boldsymbol{v} - \boldsymbol{u}\|^2_2 +  \frac12(\tilde u_D(\bar{\boldsymbol{x}}) - u_D)^2 = \frac12 \sum_{j=0}^p (v_j - u_j)^2 + \frac12\left( \sum_{j=0}^p v_j \phi_j(\bar{\boldsymbol{x}}) - u_D \right)^2.
\end{equation}

By taking the derivative of $\mathcal{L}$ with respect to $v_k$, the following optimality condition is obtained
\begin{equation}
    \frac{\partial \mathcal{L}}{\partial v_k} = (v_k - u_k) + \left( \sum_{j=0}^p v_j \phi_j(\bar{\boldsymbol{x}}) - u_D \right) \phi_k(\bar{\boldsymbol{x}}) = 0, \quad \text{for} \quad k = 0,\dots,p,
\end{equation}  
which can be written in the following matrix form
\begin{equation}\label{eq:rodEwsystem}
    \frac{\partial \mathcal{L}}{\partial \boldsymbol{v}} = \boldsymbol{v} - \boldsymbol{u} + (\boldsymbol{\phi}^T(\bar{\boldsymbol{x}}) \boldsymbol{v} - u_D) \boldsymbol{\phi}(\bar{\boldsymbol{x}}) = \boldsymbol{0}.
\end{equation}
The coefficients $\boldsymbol{v}$ are obtained by solving the linear system \eqref{eq:rodEwsystem}, and the polynomial $\tilde u_D$ in $\tilde{\boldsymbol{x}}$ can be computed. Although apparently different from previous formulations involving Lagrange multipliers, this boundary treatment can also be recast as a polynomial correction, as shown in the following result. 

\begin{proposition}
Computing the ROD-E-w modified Dirichlet boundary condition by inverting the linear system \eqref{eq:rodEwsystem} and evaluating the polynomial $\tilde u_D$ in $\tilde{\boldsymbol{x}}$ is equivalent to the following polynomial correction
\begin{equation}\label{eq:rodEw_final}
    \tilde u_D = [\boldsymbol{\phi}(\tilde{\boldsymbol{x}}) - \alpha_D^{\text{ROD-E-w}} \boldsymbol{\phi}(\bar{\boldsymbol{x}})]^T \boldsymbol{u} + \alpha_D^{\text{ROD-E-w}} u_D(\bar{\boldsymbol{x}}) , \quad \quad \text{with}
\end{equation}
\begin{equation}\label{eq:alpha_rod_e_w}
    \alpha_D^{\text{ROD-E-w}} = \frac{\boldsymbol{\phi}(\tilde{\boldsymbol{x}})^T \boldsymbol{\phi}(\bar{\boldsymbol{x}})}{1+\|\boldsymbol{\phi}(\bar{\boldsymbol{x}})\|^2_2}   .
\end{equation}
\end{proposition}
\begin{proof}
    By rearranging the previous equation, the values of $\boldsymbol{v}$ is obtained as the solution of the following linear system
\begin{equation}
    (I + \boldsymbol{\phi}(\bar{\boldsymbol{x}}) \boldsymbol{\phi}^T(\bar{\boldsymbol{x}})) \boldsymbol{v} = \boldsymbol{u} + u_D(\bar{\boldsymbol{x}}) \boldsymbol{\phi}(\bar{\boldsymbol{x}}).
\end{equation}
The inverse of the matrix $I + \boldsymbol{\phi}(\bar{\boldsymbol{x}}) \boldsymbol{\phi}^T(\bar{\boldsymbol{x}})$ can be computed by using the Sherman-Morrison formula, which gives
\begin{equation}
    (I + \boldsymbol{\phi}(\bar{\boldsymbol{x}}) \boldsymbol{\phi}^T(\bar{\boldsymbol{x}}))^{-1} = I - \frac{\boldsymbol{\phi}(\bar{\boldsymbol{x}}) \boldsymbol{\phi}^T(\bar{\boldsymbol{x}})}{1 + \boldsymbol{\phi}^T(\bar{\boldsymbol{x}}) \boldsymbol{\phi}(\bar{\boldsymbol{x}})}.
\end{equation}
Therefore, the modified degrees of freedom $\boldsymbol{v}$ can be computed as
\begin{equation}
    \begin{split}
    \boldsymbol{v} &= (I + \boldsymbol{\phi}(\bar{\boldsymbol{x}}) \boldsymbol{\phi}^T(\bar{\boldsymbol{x}}))^{-1} (\boldsymbol{u} + u_D(\bar{\boldsymbol{x}}) \boldsymbol{\phi}(\bar{\boldsymbol{x}})) \\
    &= \left( I - \frac{\boldsymbol{\phi}(\bar{\boldsymbol{x}}) \boldsymbol{\phi}^T(\bar{\boldsymbol{x}})}{1 + \|\boldsymbol{\phi}(\bar{\boldsymbol{x}})\|^2_2} \right) (\boldsymbol{u} + u_D(\bar{\boldsymbol{x}}) \boldsymbol{\phi}(\bar{\boldsymbol{x}})) \\
    &= \boldsymbol{u} + u_D(\bar{\boldsymbol{x}}) \boldsymbol{\phi}(\bar{\boldsymbol{x}}) - \frac{\boldsymbol{\phi}(\bar{\boldsymbol{x}}) \boldsymbol{\phi}^T(\bar{\boldsymbol{x}})}{1 + \|\boldsymbol{\phi}(\bar{\boldsymbol{x}})\|^2_2} (\boldsymbol{u} + u_D(\bar{\boldsymbol{x}}) \boldsymbol{\phi}(\bar{\boldsymbol{x}})) \\
    &= \boldsymbol{u} - \frac{\boldsymbol{\phi}^T(\bar{\boldsymbol{x}})\boldsymbol{u} - u_D(\bar{\boldsymbol{x}})}{1 + \|\boldsymbol{\phi}(\bar{\boldsymbol{x}})\|^2_2} \boldsymbol{\phi}(\bar{\boldsymbol{x}}) .
    \end{split}
\end{equation}

By evaluating the polynomial $\tilde u_D$ at the computational boundary point $\tilde{\boldsymbol{x}}$, we obtain
\begin{equation}
    \begin{split}
    \tilde u_D = \boldsymbol{\phi}^T(\tilde{\boldsymbol{x}}) \boldsymbol{v} = \boldsymbol{\phi}^T(\tilde{\boldsymbol{x}}) \boldsymbol{u} - \frac{\boldsymbol{\phi}^T(\bar{\boldsymbol{x}})\boldsymbol{u} - u_D(\bar{\boldsymbol{x}})}{1 + \|\boldsymbol{\phi}(\bar{\boldsymbol{x}})\|^2_2} \boldsymbol{\phi}^T(\tilde{\boldsymbol{x}}) \boldsymbol{\phi}(\bar{\boldsymbol{x}}) , 
    \end{split}
\end{equation}
which is the polynomial correction in \eqref{eq:rodEw_final}.
\end{proof}

Contrary to the previously developed polynomial corrections, the method \eqref{eq:rodEw_final} is not exact on the true boundary point $\bar{\boldsymbol{x}}$, meaning that $\tilde u_D(\bar{\boldsymbol{x}})\ne u_D$. Indeed, by simple computations, we can show that
\begin{equation}
    \begin{split}
    \tilde u_D(\bar{\boldsymbol{x}}) &= \boldsymbol{\phi}^T(\bar{\boldsymbol{x}}) \boldsymbol{v} \\
    &= \boldsymbol{\phi}^T(\bar{\boldsymbol{x}}) \boldsymbol{u} - \frac{\boldsymbol{\phi}^T(\bar{\boldsymbol{x}})\boldsymbol{u} - u_D}{1 + \|\boldsymbol{\phi}(\bar{\boldsymbol{x}})\|^2_2} \|\boldsymbol{\phi}(\bar{\boldsymbol{x}})\|^2_2 \\
    &= \left(1 - \frac{\|\boldsymbol{\phi}(\bar{\boldsymbol{x}})\|^2_2}{1 + \|\boldsymbol{\phi}(\bar{\boldsymbol{x}})\|^2_2}\right) \boldsymbol{\phi}^T(\bar{\boldsymbol{x}}) \boldsymbol{u} + \frac{\|\boldsymbol{\phi}(\bar{\boldsymbol{x}})\|^2_2}{1 + \|\boldsymbol{\phi}(\bar{\boldsymbol{x}})\|^2_2} u_D ,
    \end{split}
\end{equation}
which is a convex combination of $\boldsymbol{\phi}^T(\bar{\boldsymbol{x}}) \boldsymbol{u}$ and $u_D$ with a coefficient that is always smaller than one, and therefore it is never equal to $u_D$.

\subsubsection{ROD-\texorpdfstring{$L^2$}{L2}-w polynomial correction}

As before, a different approach can be developed by minimizing the $L^2$ norm in the cost function and adding the boundary condition as a weak constraint. In this case, the functional reads
\begin{equation}
     \mathcal{L}(\boldsymbol{v}) = \frac12 \|\tilde u_D - u_h\|^2_{L^2(\Omega)} +  \frac12(\tilde u_D(\bar{\boldsymbol{x}}) - u_D)^2,
\end{equation}
and the following optimality condition is obtained in matrix form
\begin{equation}\label{eq:rodL2wsystem}
    \frac{\partial \mathcal{L}}{\partial \boldsymbol{v}} = M(\boldsymbol{v} - \boldsymbol{u}) + (\boldsymbol{\phi}^T(\bar{\boldsymbol{x}}) \boldsymbol{v} - u_D) \boldsymbol{\phi}(\bar{\boldsymbol{x}}) = \boldsymbol{0}.
\end{equation}
The coefficients $\boldsymbol{v}$ are obtained by solving the linear system \eqref{eq:rodL2wsystem}, and the evaluated polynomial $\tilde u_D$ in $\tilde{\boldsymbol{x}}$ can be computed. This boundary treatment can also be recast as a polynomial correction, as shown next. 

\begin{proposition}
Computing the ROD-$L^2$-w modified Dirichlet boundary condition by inverting the linear system \eqref{eq:rodL2wsystem} and evaluating the polynomial $\tilde u_D$ in $\tilde{\boldsymbol{x}}$ is equivalent to the following polynomial correction
\begin{equation}\label{eq:rodL2w_final}
    \tilde u_D = [\boldsymbol{\phi}(\tilde{\boldsymbol{x}}) - \alpha_D^{\text{ROD-$L^2$-w}} \boldsymbol{\phi}(\bar{\boldsymbol{x}})]^T \boldsymbol{u} + \alpha_D^{\text{ROD-$L^2$-w}} u_D,\quad \quad \text{with} 
\end{equation}
\begin{equation}\label{eq:alpha_rod_L2_w}
    \alpha_D^{\text{ROD-$L^2$-w}} = \frac{\boldsymbol{\phi}(\tilde{\boldsymbol{x}})^T M^{-1}\boldsymbol{\phi}(\bar{\boldsymbol{x}})}{1+\|\boldsymbol{\phi}(\bar{\boldsymbol{x}})\|^2_{M^{-1}}}.
\end{equation}
\end{proposition}
\begin{proof}
    By rearranging the previous equation, the values of $\boldsymbol{v}$ can be obtained as the solution of the following linear system as
\begin{equation}
    (M + \boldsymbol{\phi}(\bar{\boldsymbol{x}}) \boldsymbol{\phi}^T(\bar{\boldsymbol{x}})) \boldsymbol{v} = M\boldsymbol{u} + u_D \boldsymbol{\phi}(\bar{\boldsymbol{x}}).
\end{equation}
The inverse of the matrix $M + \boldsymbol{\phi}(\bar{\boldsymbol{x}}) \boldsymbol{\phi}^T(\bar{\boldsymbol{x}})$ can be computed by using the generalized Sherman-Morrison formula, which gives
\begin{equation}
    (M + \boldsymbol{\phi}(\bar{\boldsymbol{x}}) \boldsymbol{\phi}^T(\bar{\boldsymbol{x}}))^{-1} = M^{-1} - \frac{M^{-1}\boldsymbol{\phi}(\bar{\boldsymbol{x}}) \boldsymbol{\phi}^T(\bar{\boldsymbol{x}})M^{-1}}{1 + \boldsymbol{\phi}^T(\bar{\boldsymbol{x}}) M^{-1}\boldsymbol{\phi}(\bar{\boldsymbol{x}})}.
\end{equation}
Therefore, the modified degrees of freedom $\boldsymbol{v}$ can be computed as
\begin{equation}
    \begin{split}
    \boldsymbol{v} &= (M + \boldsymbol{\phi}(\bar{\boldsymbol{x}}) \boldsymbol{\phi}^T(\bar{\boldsymbol{x}}))^{-1} (M\boldsymbol{u} + u_D \boldsymbol{\phi}(\bar{\boldsymbol{x}})) \\
    &= \boldsymbol{u} + u_D M^{-1} \boldsymbol{\phi}(\bar{\boldsymbol{x}}) - M^{-1}\frac{\boldsymbol{\phi}(\bar{\boldsymbol{x}}) \boldsymbol{\phi}^T(\bar{\boldsymbol{x}})M^{-1}}{1 + \|\boldsymbol{\phi}(\bar{\boldsymbol{x}})\|^2_{M^{-1}}} (M\boldsymbol{u} + u_D \boldsymbol{\phi}(\bar{\boldsymbol{x}})) \\
    &= \boldsymbol{u} - \frac{\boldsymbol{\phi}^T(\bar{\boldsymbol{x}})\boldsymbol{u} - u_D}{1 + \|\boldsymbol{\phi}(\bar{\boldsymbol{x}})\|^2_{M^{-1}}} M^{-1}\boldsymbol{\phi}(\bar{\boldsymbol{x}}) .
    \end{split}
\end{equation}

By evaluating the polynomial $\tilde u_D$ at the computational boundary point $\tilde{\boldsymbol{x}}$, we obtain:
\begin{equation}
    \begin{split}
    \tilde u_D = \boldsymbol{\phi}^T(\tilde{\boldsymbol{x}}) \boldsymbol{v} = \boldsymbol{\phi}^T(\tilde{\boldsymbol{x}}) \boldsymbol{u} - \frac{\boldsymbol{\phi}^T(\bar{\boldsymbol{x}})\boldsymbol{u} - u_D}{1 + \|\boldsymbol{\phi}(\bar{\boldsymbol{x}})\|^2_{M^{-1}}} \boldsymbol{\phi}^T(\tilde{\boldsymbol{x}}) M^{-1}\boldsymbol{\phi}(\bar{\boldsymbol{x}}) , 
    \end{split}
\end{equation}
which is the polynomial correction \eqref{eq:rodL2w_final}
\end{proof}

\subsection{The Neumann and Robin cases}\label{sec:PolyCorrGen}

In this section, we demonstrate how the general framework from the previous section can be used to develop more general embedded boundary conditions, such as Neumann and Robin. For brevity, we present only the full development of Neumann and Robin boundary conditions for the ROD-E method here; afterwards, the expressions for all $\alpha_N$ for ROD-$L^2$, ROD-E-w, and ROD-$L^2$-w are provided.

We are looking to build a modified polynomial $\tilde u_N(\boldsymbol{x})=\boldsymbol{\phi}^T(\boldsymbol{x})\boldsymbol{v}$ that satisfies the Neumann boundary condition on the true boundary point. The Neumann boundary condition is enforced by imposing the normal derivative $\boldsymbol{n}(\bar{\boldsymbol{x}}) \cdot \nabla \tilde u_N(\bar{\boldsymbol{x}}) = q_N$ on the true boundary point $\bar{\boldsymbol{x}}$. 
We refer to $\boldsymbol{n}:=\boldsymbol{n}(\bar{\boldsymbol{x}})$ as the normal to the true boundary $\bar{\boldsymbol{x}}$, while we call $\tilde{\boldsymbol{n}}:=\tilde{\boldsymbol{n}}(\tilde{\boldsymbol{x}})$ the normal to the computational boundary point $\tilde{\boldsymbol{x}}$. 
In this case, the constrained minimization problem can be solved by introducing a Lagrange multiplier $\lambda$ and by considering the following functional as
\begin{equation}
    \mathcal{L}(\boldsymbol{v} , \lambda) = \frac12 \|\boldsymbol{v} - \boldsymbol{u}\|^2_2 + \lambda (\boldsymbol{n} \cdot \nabla \tilde u_N(\bar{\boldsymbol{x}}) - q_N ) .
\end{equation}
The optimality conditions for this problem can be written as
\begin{equation}
    \frac{\partial \mathcal{L}}{\partial \boldsymbol{v}} = \boldsymbol{v} - \boldsymbol{u} + \lambda  (\boldsymbol{n} \cdot \nabla \boldsymbol{\phi}(\bar{\boldsymbol{x}})) = \boldsymbol{0} , \qquad \frac{\partial \mathcal{L}}{\partial \lambda} = (\boldsymbol{n} \cdot \nabla \boldsymbol{\phi}(\bar{\boldsymbol{x}}))^T \boldsymbol{v} - q_N = 0 , 
\end{equation}
where $\boldsymbol{n}\cdot\nabla\boldsymbol{\phi}(\bar{\boldsymbol{x}}) = [\boldsymbol{n}\cdot\nabla\varphi_1(\bar{\boldsymbol{x}}),\boldsymbol{n}\cdot\nabla\varphi_2(\bar{\boldsymbol{x}}),\ldots]^T$ is a column vector. This problem can be formulated as
\begin{equation}\label{eq:RODsystem_2D_N}
\begin{bmatrix} I & \boldsymbol{n} \cdot \nabla \boldsymbol{\phi}(\bar{\boldsymbol{x}}) \\ (\boldsymbol{n} \cdot \nabla \boldsymbol{\phi}(\bar{\boldsymbol{x}}))^T & 0 \end{bmatrix} \begin{bmatrix}\boldsymbol{v} \\ \lambda \end{bmatrix} = \begin{bmatrix} \boldsymbol{u} \\ q_N \end{bmatrix}.
\end{equation}
By solving this system, the polynomial coefficients $\boldsymbol{v}$ are obtained, and the corresponding Neumann boundary condition on the computational boundary can be computed as $\tilde q_N:=\tilde{\boldsymbol{n}} \cdot \nabla \tilde u_N(\tilde{\boldsymbol{x}})$. In the following result, we also show that this Neumann condition can be written as a polynomial correction.

\begin{proposition}
Computing the ROD-E modified Neumann boundary condition by inverting the linear system \eqref{eq:RODsystem_2D_N} and evaluating the polynomial $\tilde q_N = \tilde{\boldsymbol{n}} \cdot \nabla \tilde u_N(\tilde{\boldsymbol{x}})$ in $\tilde{\boldsymbol{x}}$ is equivalent to the following polynomial correction
\begin{equation}\label{eq:rodN_final_2D}
    \tilde q_N = [\tilde{\boldsymbol{n}} \cdot \nabla \boldsymbol{\phi}(\tilde{\boldsymbol{x}}) - \alpha_N^{\text{ROD-E}} \boldsymbol{n} \cdot \nabla \boldsymbol{\phi}(\bar{\boldsymbol{x}})]^T \boldsymbol{u} + \alpha_N^{\text{ROD-E}} q_N , \quad \quad \text{with} 
\end{equation}
\begin{equation}
    \alpha_N^{\text{ROD-E}} = \frac{(\tilde{\boldsymbol{n}} \cdot \nabla \boldsymbol{\phi}(\tilde{\boldsymbol{x}}))^T (\boldsymbol{n} \cdot \nabla \boldsymbol{\phi}(\bar{\boldsymbol{x}}))}{\|\boldsymbol{n} \cdot \nabla \boldsymbol{\phi}(\bar{\boldsymbol{x}})\|_2^2}.
\end{equation}
\end{proposition}

\begin{proof}
The proof is similar to the Dirichlet case, but we give the details for completeness. From the optimality conditions, we obtain the following relation for $\boldsymbol{v}$ and $\lambda$:
\begin{equation}
    \boldsymbol{v} = \boldsymbol{u} - \lambda  (\boldsymbol{n} \cdot \nabla \boldsymbol{\phi}(\bar{\boldsymbol{x}})), \qquad \lambda = \frac{(\boldsymbol{n} \cdot \nabla \boldsymbol{\phi}(\bar{\boldsymbol{x}}))^T \boldsymbol{u} - q_N}{(\boldsymbol{n} \cdot \nabla \boldsymbol{\phi}(\bar{\boldsymbol{x}}))^T(\boldsymbol{n} \cdot \nabla \boldsymbol{\phi}(\bar{\boldsymbol{x}}))} .
\end{equation}
In order to write it as a polynomial correction, it is necessary to consider the evaluation of the $q_N = \tilde{\boldsymbol{n}} \cdot \nabla \tilde u_N(\tilde{\boldsymbol{x}})$ in the computational boundary point $\tilde{\boldsymbol{x}}$ as
\begin{equation}
    \begin{split}
    \tilde q_N &= (\tilde{\boldsymbol{n}} \cdot \nabla \boldsymbol{\phi}(\tilde{\boldsymbol{x}}))^T \boldsymbol{v} \\
    &= (\tilde{\boldsymbol{n}} \cdot \nabla \boldsymbol{\phi}(\tilde{\boldsymbol{x}}))^T \boldsymbol{u} - \left( \frac{(\boldsymbol{n} \cdot \nabla \boldsymbol{\phi}(\bar{\boldsymbol{x}}))^T \boldsymbol{u} - q_N}{\|\boldsymbol{n} \cdot \nabla \boldsymbol{\phi}(\bar{\boldsymbol{x}})\|_2^2} \right) (\tilde{\boldsymbol{n}} \cdot \nabla \boldsymbol{\phi}(\tilde{\boldsymbol{x}}))^T (\boldsymbol{n} \cdot \nabla \boldsymbol{\phi}(\bar{\boldsymbol{x}})),
    \end{split}
\end{equation}
where $\tilde{\boldsymbol{n}}$ is the normal vector at $\tilde{\boldsymbol{x}}$. From the previous equation, we directly obtain the polynomial correction in \eqref{eq:rodN_final_2D}.

\end{proof}

Following the same path, ROD-E, ROD-$L^2$, ROD-E-w, and ROD-$L^2$-w polynomial corrections for Neumann boundary conditions can be expressed in the following unified form
\begin{equation}
    \tilde q_N = [\tilde{\boldsymbol{n}} \cdot \nabla \boldsymbol{\phi}(\tilde{\boldsymbol{x}}) - \alpha_N \,\boldsymbol{n} \cdot \nabla \boldsymbol{\phi}(\bar{\boldsymbol{x}})]^T \boldsymbol{u} + \alpha_N q_N,
\end{equation}
where the coefficient $\alpha_N$ depends on the specific method as

\begin{subequations}\label{eq:alpha_all_NBC}
\begin{align}
\alpha_N^{\text{ROD-E}} &= \frac{(\tilde{\boldsymbol{n}} \cdot \nabla \boldsymbol{\phi}(\tilde{\boldsymbol{x}}))^T (\boldsymbol{n} \cdot \nabla \boldsymbol{\phi}(\bar{\boldsymbol{x}}))}{\|\boldsymbol{n} \cdot \nabla \boldsymbol{\phi}(\bar{\boldsymbol{x}})\|_2^2}, \\[10pt]
\alpha_N^{\text{ROD-}L^2} &= \frac{(\tilde{\boldsymbol{n}} \cdot \nabla \boldsymbol{\phi}(\tilde{\boldsymbol{x}}))^T M^{-1} (\boldsymbol{n} \cdot \nabla \boldsymbol{\phi}(\bar{\boldsymbol{x}}))}{\|\boldsymbol{n} \cdot \nabla \boldsymbol{\phi}(\bar{\boldsymbol{x}})\|_{M^{-1}}^2}, \\[10pt]
\alpha_N^{\text{ROD-E-w}} &= \frac{(\tilde{\boldsymbol{n}} \cdot \nabla \boldsymbol{\phi}(\tilde{\boldsymbol{x}}))^T (\boldsymbol{n} \cdot \nabla \boldsymbol{\phi}(\bar{\boldsymbol{x}}))}{1 + \|\boldsymbol{n} \cdot \nabla \boldsymbol{\phi}(\bar{\boldsymbol{x}})\|_2^2}, \\[10pt]
\alpha_N^{\text{ROD-}L^2\text{-w}} &= \frac{(\tilde{\boldsymbol{n}} \cdot \nabla \boldsymbol{\phi}(\tilde{\boldsymbol{x}}))^T M^{-1} (\boldsymbol{n} \cdot \nabla \boldsymbol{\phi}(\bar{\boldsymbol{x}}))}{1 + \|\boldsymbol{n} \cdot \nabla \boldsymbol{\phi}(\bar{\boldsymbol{x}})\|_{M^{-1}}^2}.
\end{align}
\end{subequations}

It is interesting to notice the difference in the obtained polynomial corrections using a minimization-based approach with respect to the original shifted boundary Neumann correction reported in \eqref{eq:SBNeumannPolyCorr}, where $\alpha_N^{\text{SBM}} = \tilde{\boldsymbol{n}}\cdot\boldsymbol{n}$.

Finally, we conclude this section by generalizing all polynomial corrections developed above to Robin boundary conditions. Following the same notation, we want to build a modified polynomial $\tilde u_R(\boldsymbol{x})=\boldsymbol{\phi}^T(\boldsymbol{x})\boldsymbol{v}$ such that the following constraint is satisfied
\begin{equation}
    \tilde u_R(\bar{\boldsymbol{x}}) + \varepsilon \; \boldsymbol{n}\cdot\nabla \tilde u_R(\bar{\boldsymbol{x}}) = u_{RD} + \varepsilon q_{RN} = g_R,
\end{equation}
where $g_R$ is the exact Robin boundary condition at the true boundary point $\bar{\boldsymbol{x}}$, and $\varepsilon$ is the Robin coefficient.
For simplicity, let us call 
\begin{equation}
\boldsymbol{g}(\tilde{\boldsymbol{n}},\tilde{\boldsymbol{x}} ) = \boldsymbol{\phi}(\tilde{\boldsymbol{x}}) + \varepsilon \; \tilde{\boldsymbol{n}} \cdot \nabla \boldsymbol{\phi}(\tilde{\boldsymbol{x}}) \quad \text{and} \quad 
\boldsymbol{g}(\boldsymbol{n},\bar{\boldsymbol{x}}) = \boldsymbol{\phi}(\bar{\boldsymbol{x}}) + \varepsilon \;\boldsymbol{n} \cdot \nabla \boldsymbol{\phi}(\bar{\boldsymbol{x}}),
\end{equation}
the Robin combination of basis functions is evaluated on the surrogate boundary and on the true boundary, respectively. The modified boundary condition $\tilde g_R := \tilde u_R(\tilde{\boldsymbol{x}}) + \varepsilon \; \tilde{\boldsymbol{n}} \cdot \nabla \tilde u_R(\tilde{\boldsymbol{x}})$ in the surrogate boundary point $\tilde{\boldsymbol{x}}$ reads
\begin{equation}\label{eq:rodR_final_2D}
    \begin{split}
    \tilde g_R = [\boldsymbol{g}(\tilde{\boldsymbol{n}},\tilde{\boldsymbol{x}}) -\alpha_R \boldsymbol{g}(\boldsymbol{n},\bar{\boldsymbol{x}})]^T \boldsymbol{u} + \alpha_R g_R , 
    \end{split}
\end{equation}
where the coefficient $\alpha_R$ depends on the specific method as
\begin{subequations}
\begin{align}
\alpha_R^{\text{ROD-E}} &= \frac{\boldsymbol{g}^T(\tilde{\boldsymbol{n}},\tilde{\boldsymbol{x}}) \boldsymbol{g}(\boldsymbol{n},\bar{\boldsymbol{x}})}{\|\boldsymbol{g}(\boldsymbol{n},\bar{\boldsymbol{x}})\|_2^2}, \\[10pt]
\alpha_R^{\text{ROD-}L^2} &= \frac{\boldsymbol{g}^T(\tilde{\boldsymbol{n}},\tilde{\boldsymbol{x}}) M^{-1} \boldsymbol{g}(\boldsymbol{n},\bar{\boldsymbol{x}})}{\|\boldsymbol{g}(\boldsymbol{n},\bar{\boldsymbol{x}})\|_{M^{-1}}^2}, \\[10pt]
\alpha_R^{\text{ROD-E-w}} &= \frac{\boldsymbol{g}^T(\tilde{\boldsymbol{n}},\tilde{\boldsymbol{x}}) \boldsymbol{g}(\boldsymbol{n},\bar{\boldsymbol{x}})}{1 + \|\boldsymbol{g}(\boldsymbol{n},\bar{\boldsymbol{x}})\|_2^2}, \\[10pt]
\alpha_R^{\text{ROD-}L^2\text{-w}} &= \frac{\boldsymbol{g}^T(\tilde{\boldsymbol{n}},\tilde{\boldsymbol{x}}) M^{-1} \boldsymbol{g}(\boldsymbol{n},\bar{\boldsymbol{x}})}{1 + \|\boldsymbol{g}(\boldsymbol{n},\bar{\boldsymbol{x}})\|_{M^{-1}}^2}.
\end{align}
\end{subequations}

\section{Final variational formulation}

For the sake of completeness, we summarize the final expressions for the unknown boundary value conditions on the computational unfitted surrogate boundary, $\tilde{\Gamma}_h$, in \eqref{eq:unknown_boundary_data} in terms of the true boundary conditions on $\Gamma$ with a generic polynomial correction. Recall the general variation statement in \eqref{eq:generalized_form}, with the bi-linear and linear form on an unfitted domain given in \eqref{eq:unfitted_bilinear_form} and \eqref{eq:unfitted_linear_form}, respectively. The unknown boundary data from \eqref{eq:unknown_boundary_data} are then expressed (in a very generic symbolic form with no assumption on polynomial correction, etc.) as
\begin{subequations}\label{eq:known_boundary_data}
\begin{align}
    \tilde{u}_D    &= u_h(\tilde{\boldsymbol{x}}) -  \alpha_D u_h(\bar{\boldsymbol{x}}) + \alpha_D u_D,\\
    \tilde{q}_N    &= \tilde{\boldsymbol{n}}\cdot \boldsymbol{\nabla} u_h(\tilde{\boldsymbol{x}}) - \alpha_N \boldsymbol{n}\cdot \boldsymbol{\nabla} u_h(\bar{\boldsymbol{x}}) + \alpha_N q_N,  \\
    \tilde{g}_{R} &= u_h(\tilde{\boldsymbol{x}}) + \varepsilon \tilde{\boldsymbol{n}}\cdot \boldsymbol{\nabla} u_h(\tilde{x}) - \alpha_R(u_h(\bar{\boldsymbol{x}}) +  \varepsilon \boldsymbol{n}\cdot \boldsymbol{\nabla} u_h(\bar{\boldsymbol{x}})) + \alpha_R g_R,
\end{align}
\end{subequations}
which are then inserted in \eqref{eq:unfitted_linear_form} and rearranged to give the final variational formulation as
\begin{subequations}\label{eq:unfitted_bilinear_form_final}
\begin{align}
    a\langle u_h,w_h \rangle_{\tilde{\Omega}_h} &= (\boldsymbol{\nabla} u_h, \boldsymbol{\nabla} w_h)_{\tilde{\Omega}_h} + \beta  ( u_h,  w_h)_{\tilde{\Omega}_h},  \\
    a\langle u_h,w_h \rangle_{\tilde{\Gamma}_{h,D}} &=  \gamma^{-1} (\alpha_D u_h(\bar{\boldsymbol{x}}),w_h)_{\tilde{\Gamma}_{h,D}} - (\boldsymbol{\nabla} u_h \cdot \tilde{\boldsymbol{n}},w_h)_{\tilde{\Gamma}_{h,D}} - ( \alpha_D u_h(\bar{\boldsymbol{x}}),\boldsymbol{\nabla} w_h \cdot \tilde{\boldsymbol{n}})_{\tilde{\Gamma}_{h,D}},  \\
    a \langle u_h,w_h \rangle_{\tilde{\Gamma}_{h,N}} &= (\alpha_N \boldsymbol{n} \cdot \boldsymbol{\nabla} u_h(\bar{\boldsymbol{x}}) - \tilde{\boldsymbol{n}}\cdot \boldsymbol{\nabla} u_h,w_h)_{\tilde{\Gamma}_{h,N}}, \\
    a \langle u_h,w_h \rangle_{\tilde{\Gamma}_{h,R}} &= \frac{1}{\varepsilon + \gamma} ( \alpha_R(u_h(\bar{\boldsymbol{x}}) +  \varepsilon \boldsymbol{n}\cdot \boldsymbol{\nabla} u_h(\bar{\boldsymbol{x}})) - \varepsilon \tilde{\boldsymbol{n}}\cdot \boldsymbol{\nabla} u_h,w_h)_{\tilde{\Gamma}_{h,R}} 
     \notag \\
    & - \frac{\gamma}{\varepsilon + \gamma} (\alpha_R(u_h(\bar{\boldsymbol{x}}) +  \varepsilon \boldsymbol{n}\cdot \boldsymbol{\nabla} u_h(\bar{\boldsymbol{x}})) - \varepsilon \tilde{\boldsymbol{n}}\cdot \boldsymbol{\nabla} u_h,\boldsymbol{\nabla} w_h \cdot \tilde{\boldsymbol{n}})_{\tilde{\Gamma}_{h,R}} \notag \\
    & - \frac{\gamma}{\varepsilon + \gamma} (\boldsymbol{\nabla} u_h \cdot \tilde{\boldsymbol{n}},w_h)_{\tilde{\Gamma}_{h,R}} - \frac{\varepsilon\gamma}{\varepsilon + \gamma} (\boldsymbol{\nabla} u_h \cdot \tilde{\boldsymbol{n}},\boldsymbol{\nabla} w_h \cdot \tilde{\boldsymbol{n}})_{\tilde{\Gamma}_{h,R}} ,   
\end{align}
\end{subequations}
and
\begin{subequations}\label{eq:unfitted_linear_form_final}
\begin{align} 
    b\langle w_h \rangle_{\tilde{\Omega}_h} &= (f,w_h)_{\tilde{\Omega}_h},  \\
    b\langle w_h \rangle_{\tilde{\Gamma}_{h,D}} &=  \gamma^{-1} (\alpha_D u_D,w_h)_{\tilde{\Gamma}_{h,D}} 
    -(\alpha_D u_D,\boldsymbol{\nabla} w_h \cdot \tilde{\boldsymbol{n}})_{\tilde{\Gamma}_{h,D}} ,  \\
    b \langle w_h \rangle_{\tilde{\Gamma}_{h,N}} &= (\alpha_Nq_N,w_h)_{\tilde{\Gamma}_{h,N}},  \\
    b\langle w_h \rangle_{\tilde{\Gamma}_{h,R}} &=  \frac{1}{\varepsilon + \gamma} (\alpha_R g_R,w_h)_{\tilde{\Gamma}_{h,R}} - \frac{\gamma}{\varepsilon + \gamma} (\alpha_R g_{R},\boldsymbol{\nabla} w_h \cdot \tilde{\boldsymbol{n}})_{\tilde{\Gamma}_{h,R}}.
\end{align}
\end{subequations}

\section{Results in 1D}\label{sec:results_1D}

As an initial investigation, we examine numerical results in 1D. The computational domain is built with shifted boundaries at the left and right ends, separated by a distance $d$. Note that $d=0$ corresponds to the conformal case, and for example, $d/h=1$ indicates a distance of one element length.

\subsection{Visualization of the polynomial corrections and their boundedness}

First, we visualize the extrapolated nodal Lagrangian basis function in the reference domain, $\mathcal{R}$, given by $\phi(r)$, for different polynomial orders, $p \in \{1,...,4\}$. Consider Section \ref{sec:SEM} to recall the notation. We zoom in on the leftmost element and consider only the essential boundary condition for simplicity. Additionally, we show the shape of the polynomial correction, denoted as $\mathcal{P} = \phi - \alpha\phi|_{x=\bar{x}}$, where $\tilde{u}_D = \mathcal{P}u_h + \alpha u_D$ in a general sense. We do this for the original polynomial correction, SBM, and for the minimization-based approaches ROD-E and ROD-$L^2$, as visualized in Figure \ref{fig:basis_functions} for a shifted boundary distance of $d/h=1$. In the first column, the nodal basis functions are shown inside the element (solid lines) and extrapolated outside it (dashed lines), highlighting unbounded behavior when evaluating them outside the element, especially at high orders. Now, considering the polynomial correction in $\mathcal{R}$, $\mathcal{P}(r)$ for SBM in the middle column, the imposed correction, $\mathcal{P}(r=-1)$, behaves similarly unbounded. Looking at ROD-E and ROD-$L^2$, in the right column, we clearly see how the minimized-based polynomial correction i) mimics the internal Lagrangian polynomial and ii) remains bounded even at higher polynomial orders. The same behavior is observed for $p\geq 5$ and for ROD-E-w and ROD-$L^2$-w, but it is not shown here for the sake of brevity.

To elaborate further on the behavior of $\max |\mathcal{P}(r=-1)|$, we extend the analysis to a range of $d/h \in [0,2]$ for polynomial order $p \in \{1,\ldots,6\}$ as shown in Figure \ref{fig:DBC_maximum_basis}. This again confirms the unbounded trend of the original polynomial correction (SBM) and the boundedness of the minimization-based approaches (ROD-E and ROD-$L^2$).

\newpage

\begin{figure}
    \centering
    \includegraphics[scale=0.8]{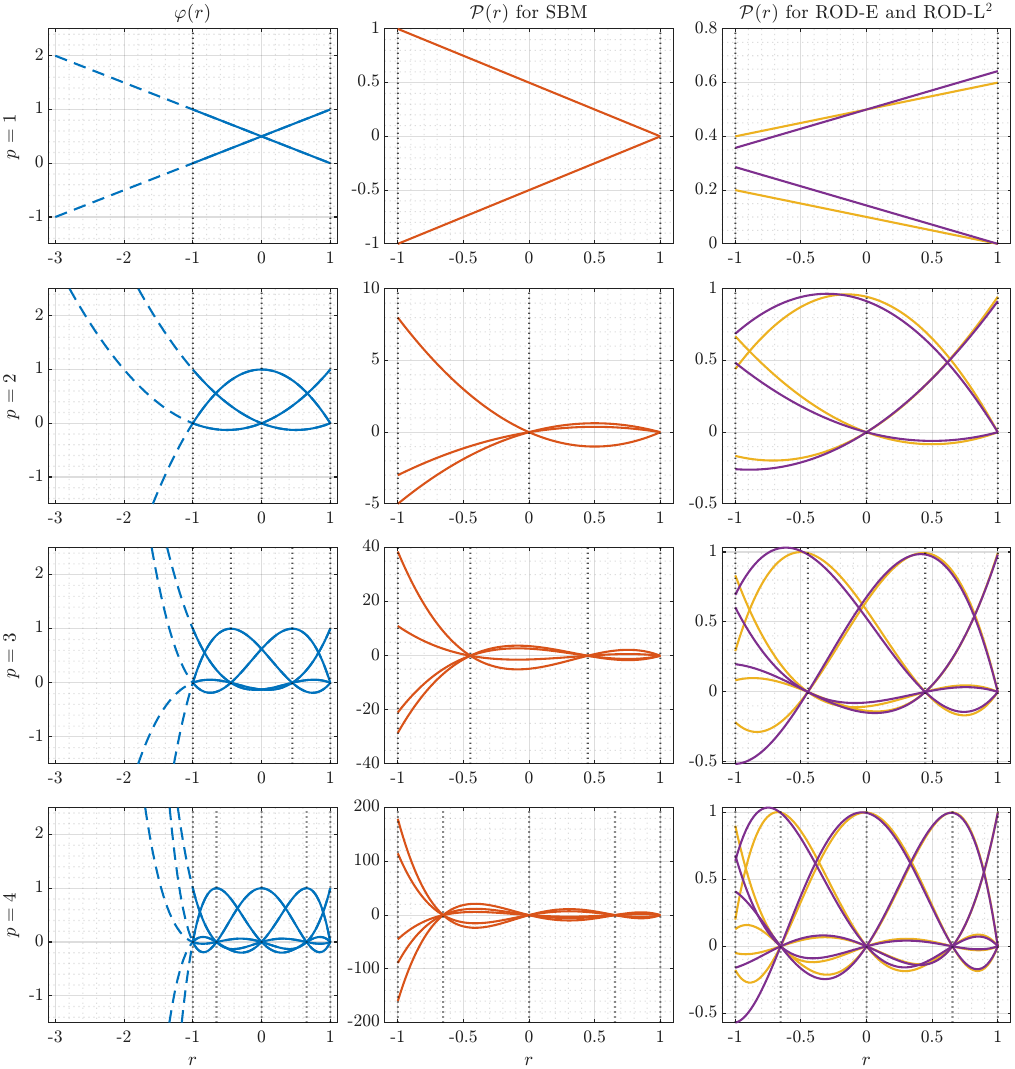}
    \caption{Left column: Nodal basis functions, $\phi(r)$, inside (solid lines) and extrapolated (dashed lines) outside the element for different polynomial orders.
    Middle column: Shape of the polynomial correction, $\mathcal{P}(r)$, for the original SBM polynomial correction. Recall that $\mathcal{P}(r=-1)$ is the imposed correction.
    Right column: Shape of $\mathcal{P}(r)$ using the minimization-based polynomial corrections approaches of ROD-E (yellow) and ROD-$L^2$ (purple).}
    \label{fig:basis_functions}
\end{figure}

\begin{figure}
    \centering
    \includegraphics[scale=0.80]{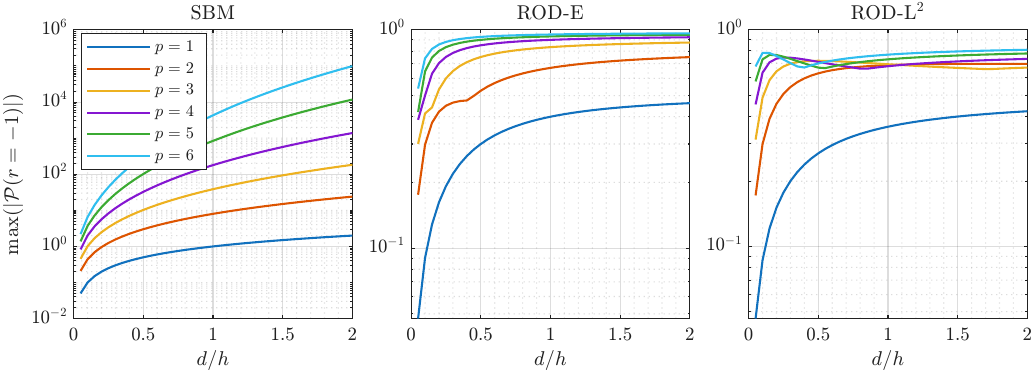}
    \caption{Evaluation of $\max |\mathcal{P}(r=-1)|$ for SBM, ROD-E, and ROD-L$^2$ with different $p$ and $d/h$.}
    \label{fig:DBC_maximum_basis}
\end{figure}

\newpage

\subsection{Convergence study for Dirichlet and Neumann boundary conditions}

Next, we consider a classical $h$-convergence study of fixed polynomial orders, $p \in \{2,4,6\}$ and report the metrics, $\|u-u_h\|_{\infty}$, $\|Au-b\|_{\infty}$, and $\kappa(A)$. Here, $u$ is a manufactured solution to the Poisson problem in \eqref{eq:Poisson_strong} with high-polynomial, asymmetric trigonometric, and exponential features, i.e., $u = u(x) = e^x + x^9 + \cos(x 2.3 \pi + 0.3)$ for $x \in [0,1]$. The numerical solution, $u_h$, is interpolated to a finer equidistant grid, whereon the error is evaluated, to circumvent any favorable super-convergent point selection. Moreover, $A$ and $b$ are the arising system matrix and vector from the bilinear and linear form in \eqref{eq:generalized_form}, respectively. The sparse linear system of equations, $A u_h = b$, is solved for $u_h$, using a classical Gaussian elimination procedure in Matlab using the \textit{backslash} operator. At last, $\kappa(A)$ measures the conditioning of the system matrix in terms of the condition number, which is important for the future use of efficient iterative solvers without the inclusion of stabilization terms to circumvent slow or unstable convergence. For this test, we consider the Poisson problem subject to Dirichlet and Neumann boundary conditions (one at a time), but, for the sake of brevity, we reserve the general Robin conditions for the 2D results. The 1D computational domain is constructed using $N_x = \{4,8,16,32\}$ elements, i.e., halving the domain at each increment, with a shifted boundary distance of $d/h = 1$.

The results for the Dirichlet boundary conditions can be seen in Figure \ref{fig:h_convergence_1D_DBC} and for Neumann boundary conditions in Figure \ref{fig:h_convergence_1D_NBC}. Ultimately, the expected error convergence rate in the solution, $\|u-u_h\|_{\infty}$, is observed, i.e., $\mathcal{O}(h^{p+1})$ (optimal) and $\mathcal{O}(h^{p})$ (sub-optimal) for Dirichlet and Neumann conditions, respectively. All methods provide comparable errors. Note that the sub-optimality of the Neumann case is due to the approximation accuracy of the first derivatives in \eqref{eq:Wp}. Moreover, $\mathcal{O}(h^{p})$ convergence rate is observed for $\|Au-b\|_{\infty}$ in both cases, with SBM providing significantly higher errors than all minimization-based methods. That is a direct consequence of matrix conditioning, as confirmed by considering $\kappa(A)$, where SBM provides significantly larger (multiple orders of magnitude) values than the minimization-based ones. In general, the conditioning scales with $\mathcal{O}(h^{-2})$, but for Dirichlet boundary conditions on coarse meshes with higher orders, $p \geq 4$, the scaling is only $\mathcal{O}(h^{-1})$. Similar non-$\mathcal{O}(h^{-2})$-scaling for high-orders was observed in \cite{gorgi2026isogeometric}, posing an \textit{open question}. One explanation could be that the condition number is governed in three regions: i) a classical asymptotic $\mathcal{O}(h^{-2})$ scaling region, ii) a $\mathcal{O}(d)$ dominated region, and iii) an insufficiently resolved region. Both i) and iii) are classical regions in boundary-fitted methods, but ii) is only valid for unfitted approaches and might also depend on the approximation order, $p$. Recall that in Figure \ref{fig:h_convergence_1D_DBC} and Figure \ref{fig:h_convergence_1D_NBC}, $d = h$. In general, we observe that reducing $d/h$ shifts the transition from ii) to i) occurring on coarser meshes. For $d/h = 0$, ii) has disappeared completely. Moreover, we observe that for a fixed $d/h$, increasing $p$ shifts the transition to finer meshes. This \textit{delayed asymptotic behavior} will be further examined in the 2D results, where $d/h$ is much more complicated. In general, all curves are converging in a smooth and asymptotic manner, due to the fixed $d/h$ quantity. This feature will be challenged when moving to 2D, as $d/h$ is uncontrollable and governed by the background mesh and the location of the true boundary. Lastly, ROD-E and ROD-E-w are visually similar. The same can be said about ROD-$L^2$ and ROD-$L^2$-w. Recall the formula of $\alpha$ from \eqref{eq:alpha_rod_e}, \eqref{eq:alpha_rod_L2}, \eqref{eq:alpha_rod_e_w}, and \eqref{eq:alpha_rod_L2_w} in the Dirichlet case and \eqref{eq:alpha_all_NBC} in the Neumann case. Here, the difference between the strong and the weak constrained formulations is an additional identity in the denominator, which, for the 1D case, is close to negligible. In 2D, however, this will not be the case, as will be shown later.

\textbf{Remark on effect of mapping and element selection}: As noted in Section \ref{sec:intro}, the high-order SBM polynomial correction work in \cite{visbech2025spectral} was motivated to improve the accuracy and conditioning of the original SBM by changing i) the element selection as discussed in Section \ref{sec:unfitted_comp_domain} and ii) the mapping in \eqref{eq:mapping}. One might question whether such an approach could yield further improvement of the work presented in this paper. Here, the authors argue (not shown, but based on experience) that no to minor improvements are to be expected since the methodology in \cite{visbech2025spectral} effectively bounds the polynomial correction inside the elements by choosing i) and ii) in a clever way. Similarly, in this work, we ensure that the polynomial correction (based on minimization) is bounded as showcased in Figure \ref{fig:basis_functions}. Thus, no additional benefit from bounding an already bounded polynomial.

\newpage

\begin{figure}
    \centering
    \includegraphics[scale=0.85]{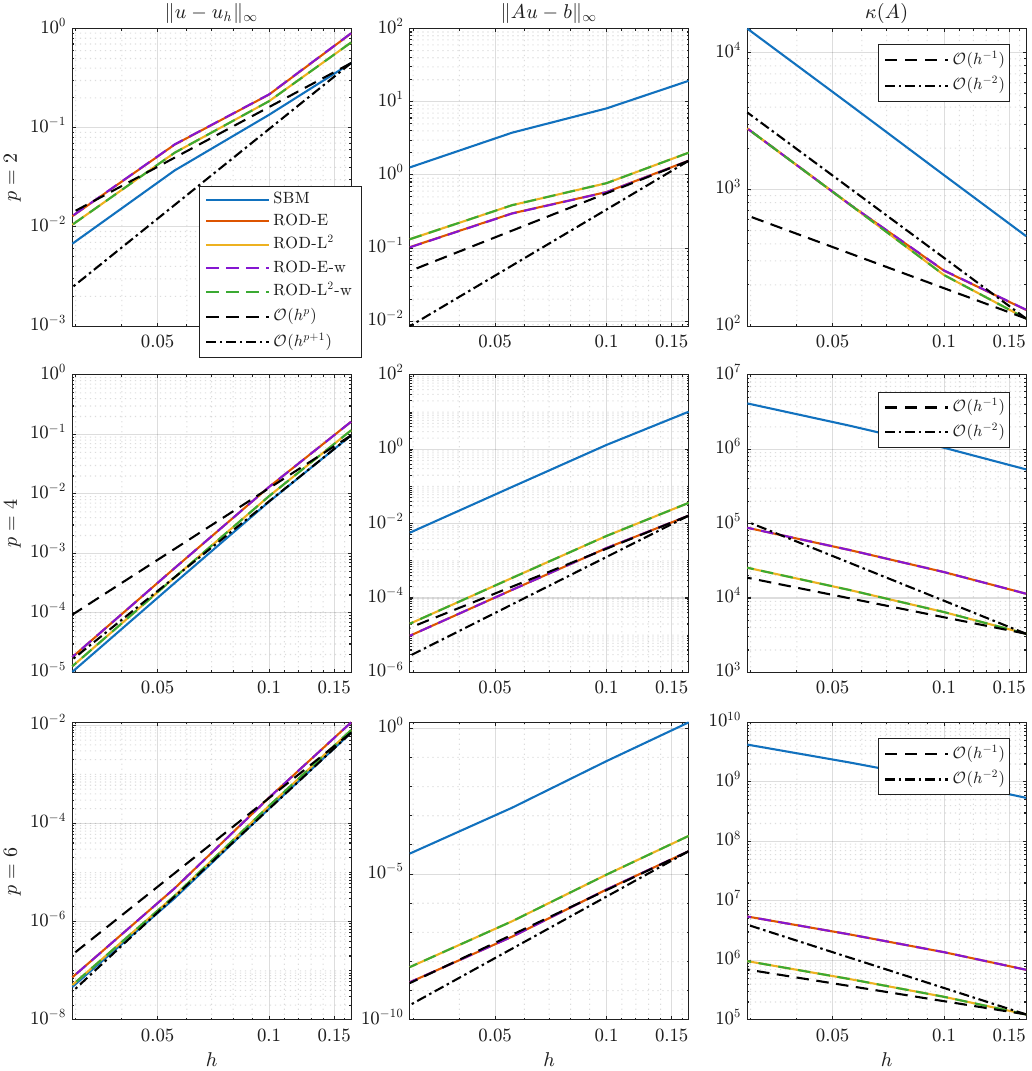}
    \caption{1D mesh refinement convergence study for $p\in\{2,4,6\}$ with Dirichlet boundary conditions and $d/h=1$.}
    \label{fig:h_convergence_1D_DBC}
\end{figure}

\newpage

\begin{figure}
    \centering
    \includegraphics[scale=0.85]{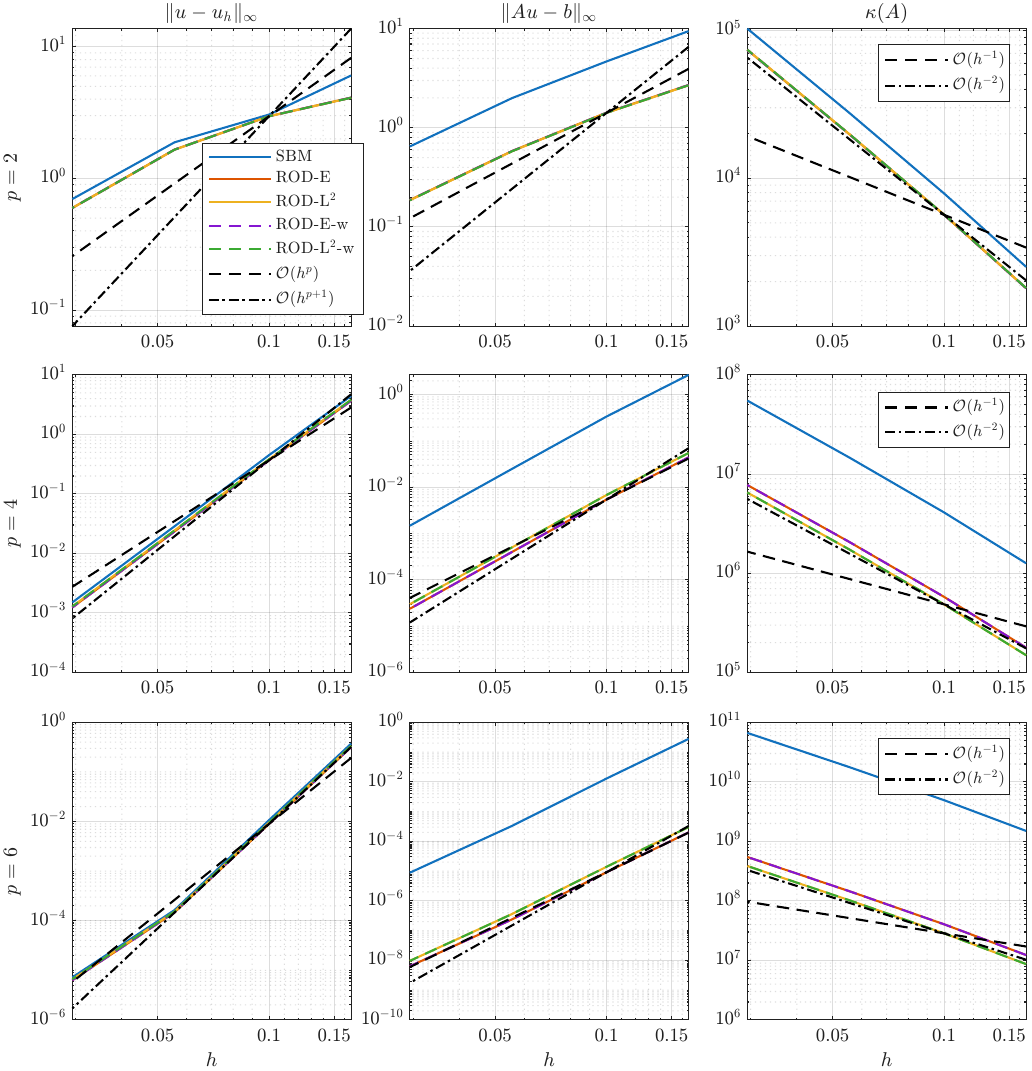}
    \caption{1D $h$-convergence study for $p\in\{2,4,6\}$ with Neumann boundary conditions and $d/h=1$.}
    \label{fig:h_convergence_1D_NBC}
\end{figure}

\section{Results in 2D}\label{sec:results_2D}

Next, we consider results in 2D. For this, we consider uniform background mesh of $N_x = N_y = N$ quadrilaterals for $0 \leq (x,y) \leq 2$. At $(x_0,y_0) = (1,1)$ a circle of radius $R = 0.5$ is located. As an manufactured verification solution, we consider $u = u(x,y) = cos(2\pi r^2)$, where $r = (x-x_0) + (y-y_0)$. Four different handpicked meshes using $N \in \{8,16,32,64\}$ are showcased in Figure \ref{fig:four_meshes}. As in the 1D case, errors, $\|u-u_h\|$, are interpolated (with matching high-order precision) to a finer, equidistant mesh, whereon the error is computed to eliminate possible super-convergent features from the GLL-point distribution.

\begin{figure}
    \centering
    \includegraphics[scale=0.75]{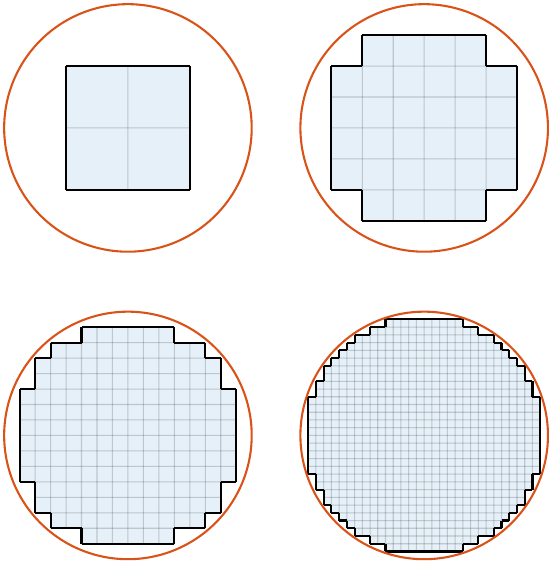}
    \caption{Four unfitted meshes of a circular domain with different resolutions of the background mesh with the notation of Figure \ref{fig:unfitted_domain_illustration}. Top-left: $N = 8$. Top-right: $N = 16$. Bottom-left: $N = 32$. Bottom-right: $N = 64$. }
    \label{fig:four_meshes}
\end{figure}

\subsection{Convergence study for Dirichlet and Neumann boundary conditions}

First, we consider the $h$-convergence study of the Poisson problem subject to Dirichlet and Neumann boundary conditions. We consider the order of the polynomial basis functions as $p \in \{1, \ldots,6\}$. As noted in, e.g., \cite{antonelli2026isogeometric}, the classical way of \textit{halving} the domain, i.e., $N \in \{2,4,8,\ldots\}$, when performing convergence studies could introduce a biased mesh configuration (either in- or unfavorable with respect to the convergence of any numerical metric). With that in mind, we defined uniform background meshes with $N_x = N_y = N \in \{8,9, \ldots,127,128\}$, i.e., successively increasing the resolution by one. We perform the mesh-refinement study in the following metrics: $\|u-u_h\|_{L^1}$, $\|Au-b\|_{L^1}$, and $\kappa(A)$. This, to establish a reference, using only SBM and ROD (see Section \ref{sec:SBcorrection} and \ref{sec:ROD}, respectively) with Dirichlet boundary conditions in Figure \ref{fig:h_convergence_2D_SBM_ROD_DBC} and Neumann boundary conditions in Figure \ref{fig:h_convergence_2D_SBM_ROD_NBC}, about which we can compare the other minimization-based approaches later. From the results, we fit a linear trend to estimate the asymptotic convergence rate for values with $h \leq 0.1$.

\begin{figure}
    \centering
    \includegraphics[scale=0.85]{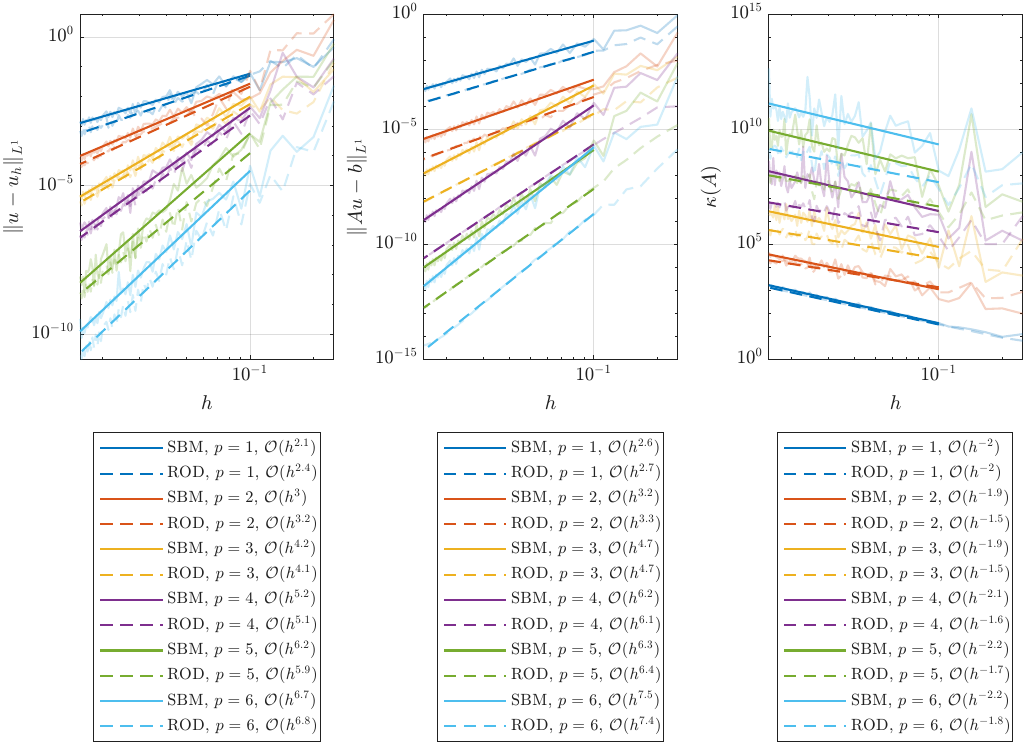}
    \caption{2D mesh refinement study for $p \in \{1, \ldots,6\}$ with Dirichlet conditions for SBM (solid lines) and ROD (dashed lines). Transparent lines are errors due to successive mesh refinement, and non-transparent lines are linear fits for errors below $h \leq 0.1$.}
    \label{fig:h_convergence_2D_SBM_ROD_DBC}
\end{figure}

From the convergence study of $\|u-u_h\|_{L^1}$, we confirm the expected optimal, $\mathcal{O}(h^{p+1})$, rate for pure Dirichlet problems, and the sub-optimal, $\mathcal{O}(h^{p})$ for pure Neumann problems by considering the linear fit values, $\mathcal{O}(h^m)$. Also, we note the highly mesh-dependent error level (transparent lines), oscillating from one mesh to the next; however, it remains convergent. This observation is well known in high-order unfitted methodologies. Moving the true geometry slightly on the same background mesh will, without a doubt, alter the solution. This phenomenon can be observed, as mentioned, not only when moving the object but also during successive mesh refinement \cite{gorgi2026isogeometric, antonelli2026isogeometric}. In the SBM setting, the weighted-SBM was proposed to smooth the time-dependent pressure signal of a moving cylinder \cite{colomes2021weighted}, and the Gap-SBM and high-order IGA Gap-SBM were proposed to mitigate this issue as well \cite{collins2026gap, gorgi2026isogeometric}. Circumventing this feature is not currently the focus of the work, yet we emphasize its existence. In general, the ROD method yields slightly lower level errors (between zero and one order of magnitude) in the solution variable compared with SBM. Moreover, the mesh-dependent error oscillations look to be more pronounced for SBM than for ROD, as a consequence of the minimization-based construction of the extended polynomial basis. 
Now, considering $\|Au-b\|_{L^1}$, we have a metric uninfluential by conditioning, which is clearly seen in the reduced oscillatory behavior of the errors under mesh-refinement, especially for the ROD method. Both approaches converge to optimality and are even slightly super-convergent in some cases. As shown in 1D, we observed lower errors for ROD, especially as $p$ increases, compared to SBM. 
Lastly, the conditioning in terms of $\kappa(A)$ yields about second order for the SBM. As for the ROD, two improvements are observed compared to SBM: i) multiple orders lower condition number for high orders, $p \geq 3$, and ii) improved scaling. On the second, we extend the discussion from 1D by highlighting that we again observe this above-optimal scaling, as pointed out in \cite{gorgi2026isogeometric}.

\begin{figure}
    \centering
    \includegraphics[scale=0.85]{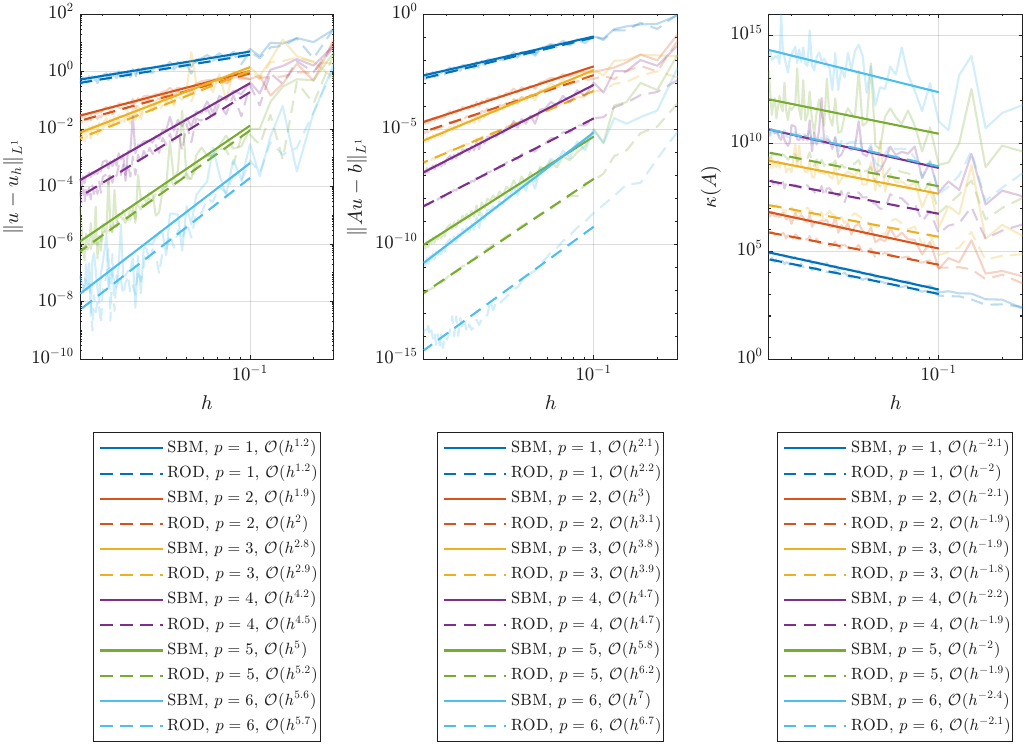}
    \caption{2D mesh refinement study for $p \in \{1, \ldots,6\}$ with Neumann conditions for SBM (solid lines) and ROD (dashed lines). Transparent lines show errors due to successive mesh refinement, and non-transparent lines are linear fits for errors below $h \leq 0.1$.}
    \label{fig:h_convergence_2D_SBM_ROD_NBC}
\end{figure}

\textbf{Remark on use of element types:} The authors note that using structured quadrilaterals instead of unstructured triangles, as presented in \cite{visbech2025spectral}, poses a reduction in the oscillatory behavior of various numerical quantities during successive mesh refinement (not shown). In fact, we argue that in the context of unfitted methods, such as the SBM, where the problem is solved on a computational domain generated from a background mesh, employing fully structured quadrilaterals in 2D and hexahedra in 3D is preferable. This approach effectively decouples mesh generation from the domain geometry. If higher-resolution regions are necessary, non-conformal mesh refinement strategies can be employed; see, e.g., \cite{yang2024optimal, gorgi2026isogeometric}.

\subsubsection{Comparing to minimization-based polynomial corrections}

Next, we extend the convergence study to include the minimization-based polynomial corrections, ROD-E, ROD-$L^2$, ROD-E-w, and  ROD-$L^2$-w. For this and for simplicity, we report only the estimates of the asymptotic convergence rate and compare them with the recently established ones for SBM and ROD for $p \in \{2,4,6\}$. Results for Dirichlet boundary conditions are shown in Figure \ref{fig:h_convergence_2D_all_DBC} and in Figure \ref{fig:h_convergence_2D_all_NBC} the Neumann case is reported. In short, the pattern is clear: all minimization-based polynomial corrections i) converge to optimality or sub-optimality as SBM and ROD, ii) mimic, to a large extent, the features of ROD (as expected) rather than SBM, iii) are somewhat bounded by ROD.

\begin{figure}
    \centering
    \includegraphics[scale=0.85]{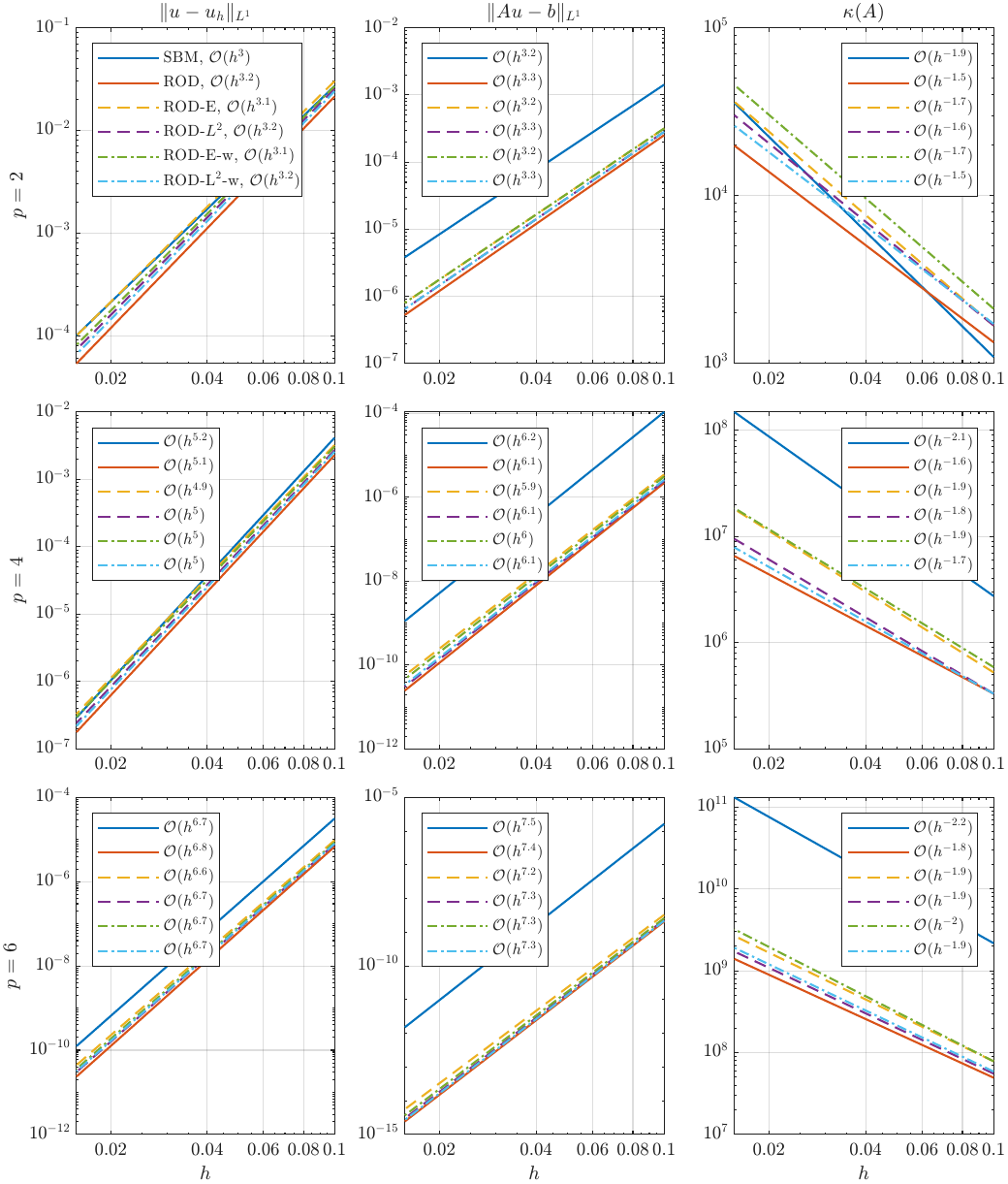}
    \caption{Complete 2D mesh-refinement study for all minimization-based polynomial corrections, SBM, and ROD for $p \in \{2,4,6\}$ in terms of estimated fits of the asymptotic convergence rate. Subject to Dirichlet boundary conditions.}
    \label{fig:h_convergence_2D_all_DBC}
\end{figure}

\begin{figure}
    \centering
    \includegraphics[scale=0.85]{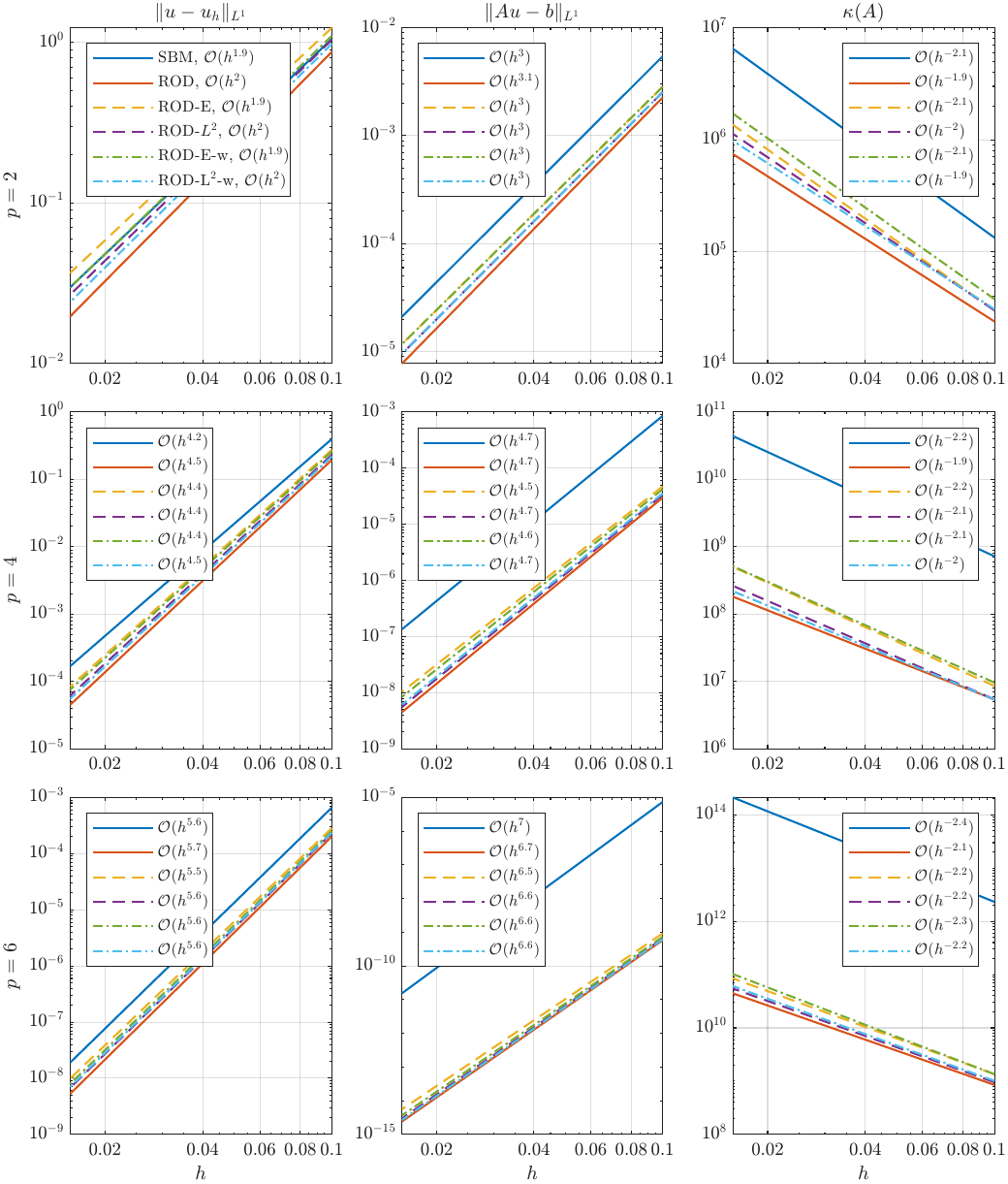}
    \caption{Complete 2D mesh-refinement study for all minimization-based polynomial corrections, SBM, and ROD for $p \in \{2,4,6\}$ in terms of estimated fits of the asymptotic convergence rate. Subject to Neumann boundary conditions.}
    \label{fig:h_convergence_2D_all_NBC}
\end{figure}

\subsection{Robin boundary conditions}

As a final numerical result, we consider the case in which the Poisson problem is subject to Robin boundary conditions. For this verification and for brevity, we solely consider the metric $\|u-u_h\|_{L^1}$ for $p\in\{4,5,6\}$ on four different background meshes with $N\in\{8,16,32,64\}$ as depicted in Figure \ref{fig:four_meshes}. Moreover, we run using the ROD-$L^2$-w polynomial correction. We run the usual convergence study for $\varepsilon \in \{0.1,1,10\}$ as shown in Figure \ref{fig:h_convergence_2D_RBC}. Convergence is confirmed; however, the finest mesh for $p=6$ might suffer from the aforementioned and shown unfavorable bias. Moreover, as expected, decreasing $\varepsilon$, i.e., mimicking Dirichlet boundary conditions, will decrease the error level.

\begin{figure}
    \centering
    \includegraphics[scale = 0.85]{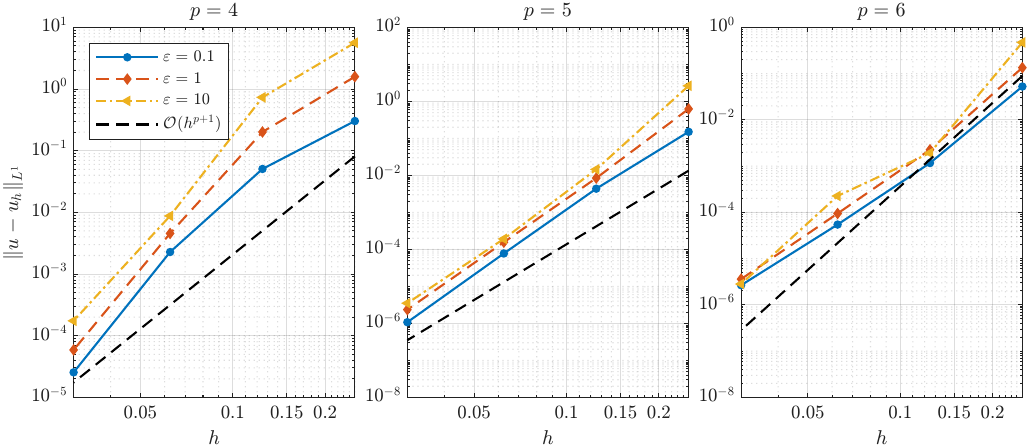}
    \caption{Mesh-refinement study for Robin conditions using ROD-$L^2$-w for $p\in\{4,5,6\}$ and $\varepsilon \in \{0.1,1,10\}$.}
    \label{fig:h_convergence_2D_RBC}
\end{figure}
\section{Conclusion}\label{sec:conclusion}

In this paper, we introduced a new family of generalized minimization-based polynomial corrections for imposing high-order unfitted (embedded) boundary conditions for the Poisson problem discretized with a spectral element method. The proposed approach generalizes the formulation of \cite{ciallella2023shifted, visbech2025spectral} within the shifted boundary method (SBM) while incorporating key ideas from the ROD method. By reformulating the original ROD minimization problem, the resulting boundary treatments can be expressed as simple polynomial corrections, eliminating the need to solve local linear systems. We showed that the proposed corrections differ from the original SBM polynomial correction only by a scaling factor, yet this seemingly minor modification fundamentally changes the method's behavior. A one-dimensional analysis demonstrated that the resulting basis functions remain bounded and as close as possible to their original counterparts, whereas the standard SBM basis functions become unbounded as the distance between the true and surrogate boundaries increases, particularly for higher polynomial orders. This significantly improves the conditioning of the discrete system with virtually no additional computational cost. The methodology was further extended from Dirichlet to Neumann and Robin boundary conditions, and numerical experiments confirmed that the resulting generalized polynomial corrections retain high-order accuracy on unfitted meshes.\\ \newline 
Future work includes extending the proposed corrections to time-dependent and nonlinear PDEs, such as the Navier-Stokes equations, as well as to problems involving moving interfaces, including multiphase flows and fluid-structure interaction \cite{monasse2012conservative,desmons2024fully}.

\subsection*{Acknowledgments}
Parts of the research were carried out at the Technical University of Denmark (DTU) in the Department of Applied Mathematics and Computer Science (Compute), with computational resources provided by the DTU Computing Center.

\subsection*{Declarations}
The authors declared that they had no potential conflicts of interest with respect to the research, authorship, and/or publication of this article. Generative AI was used solely to improve the grammar, spelling, and readability of this manuscript. The authors reviewed and edited all AI-assisted suggestions and take full responsibility for the final content.

\subsection*{Data availability}
Data can be made available on request.

\subsection*{Funding}
This work contributes to the activities of JV's PhD project, which is fully funded by DTU Compute.

\subsection*{CRediT authorship contribution statement}
Both authors contributed equally to this manuscript. More specifically: \textbf{Mirco Ciallella}: Conceptualization, Formal Analysis, Investigation, Methodology, Writing – Original Draft, Writing – Review \& Editing. \textbf{Jens Visbech}: Investigation, Methodology, Software, Validation, Visualization, Writing – Original Draft, Writing – Review \& Editing. 

\bibliographystyle{abbrvnat}
\bibliography{references.bib}  

@article{collins2026gap,
  title={{G}ap-{SBM}: A new conceptualization of the shifted boundary method with optimal convergence for the {N}eumann and {D}irichlet problems},
  author={Collins, J Haydel and Li, Kangan and Lozinski, Alexei and Scovazzi, Guglielmo},
  journal={Computer Methods in Applied Mechanics and Engineering},
  volume={453},
  pages={118793},
  year={2026},
  publisher={Elsevier}
}

@book{zienkiewicz1977finite,
  title={The finite element method},
  author={Zienkiewicz, Olgierd Cecil and Taylor, Robert Leroy and Nithiarasu, Perumal and Zhu, JZ},
  year={1977},
  publisher={Elsevier}
}

@article{colomes2021weighted,
  title={A weighted shifted boundary method for free surface flow problems},
  author={Colom{\'e}s, Oriol and Main, Alex and Nouveau, L{\'e}o and Scovazzi, Guglielmo},
  journal={Journal of Computational Physics},
  volume={424},
  pages={109837},
  year={2021},
  publisher={Elsevier}
}

@article{patera1984spectral,
  title={A spectral element method for fluid dynamics: laminar flow in a channel expansion},
  author={Patera, Anthony T},
  journal={Journal of Computational Physics},
  volume={54},
  number={3},
  pages={468--488},
  year={1984},
  publisher={Elsevier}
}

@article{ciallella2025minimization,
  title={Minimization-based embedded boundary methods as polynomial corrections: a stability study of discontinuous {G}alerkin for hyperbolic equations},
  author={Ciallella, Mirco},
  journal={arXiv preprint, 2512.05278},
  year={2025}
}

@article{scardovelli1999direct,
  title={Direct numerical simulation of free-surface and interfacial flow},
  author={Scardovelli, Ruben and Zaleski, St{\'e}phane},
  journal={Annual Review of Fluid Mechanics},
  volume={31},
  number={1},
  pages={567--603},
  year={1999},
  publisher={Annual Reviews 4139 El Camino Way, PO Box 10139, Palo Alto, CA 94303-0139, USA}
}

@article{tam2004computational,
  title={Computational aeroacoustics: An overview of computational challenges and applications},
  author={Tam, Christopher KW},
  journal={International Journal of Computational Fluid Dynamics},
  volume={18},
  number={6},
  pages={547--567},
  year={2004},
  publisher={Taylor \& Francis}
}

@article{imam1982three,
  title={Three-dimensional shape optimization},
  author={Imam, M Hasan},
  journal={International Journal for Numerical Methods in Engineering},
  volume={18},
  number={5},
  pages={661--673},
  year={1982},
  publisher={Wiley Online Library}
}

@article{beaugendre2003fensap,
  title={{FENSAP-ICE}'s three-dimensional in-flight ice accretion module: {ICE3D}},
  author={Beaugendre, H{\'e}lo{\"\i}se and Morency, Fran{\c{c}}ois and Habashi, Wagdi G},
  journal={Journal of Aircraft},
  volume={40},
  number={2},
  pages={239--247},
  year={2003}
}

@article{taylor2009patient,
  title={Patient-specific modeling of cardiovascular mechanics},
  author={Taylor, Charles A and Figueroa, CA},
  journal={Annual Review of Biomedical Engineering},
  volume={11},
  number={1},
  pages={109--134},
  year={2009},
  publisher={Annual Reviews}
}

@article{ciallella2025stability,
  title={Stability analysis of discontinuous {G}alerkin with a high order embedded boundary treatment for linear hyperbolic equations},
  author={Ciallella, Mirco},
  journal={arXiv preprint, 2510.23231},
  year={2025}
}

@article{santos2024very,
  title={{Very high-order accurate discontinuous {G}alerkin method for curved boundaries with polygonal meshes}},
  author={Santos, Milene and Ara{\'u}jo, Ad{\'e}rito and Barbeiro, S{\'\i}lvia and Clain, St{\'e}phane and Costa, Ricardo and Machado, Gaspar J},
  journal={Journal of Scientific Computing},
  volume={100},
  number={3},
  pages={66},
  year={2024},
  publisher={Springer}
}

@article{costa2018very,
  title={Very high-order accurate finite volume scheme on curved boundaries for the two-dimensional steady-state convection--diffusion equation with {D}irichlet condition},
  author={Costa, Ricardo and Clain, St{\'e}phane and Loub{\`e}re, Rapha{\"e}l and Machado, Gaspar J},
  journal={Applied Mathematical Modelling},
  volume={54},
  pages={752--767},
  year={2018},
  publisher={Elsevier}
}

@article{costa2019very,
  title={Very high-order accurate finite volume scheme for the convection-diffusion equation with general boundary conditions on arbitrary curved boundaries},
  author={Costa, Ricardo and N{\'o}brega, Jo{\~a}o M and Clain, St{\'e}phane and Machado, Gaspar J and Loub{\`e}re, Rapha{\"e}l},
  journal={International Journal for Numerical Methods in Engineering},
  volume={117},
  number={2},
  pages={188--220},
  year={2019},
  publisher={Wiley Online Library}
}

@article{ciallella2023shifted,
  title={Shifted boundary polynomial corrections for compressible flows: high order on curved domains using linear meshes},
  author={Ciallella, Mirco and Gaburro, Elena and Lorini, Marco and Ricchiuto, Mario},
  journal={Applied Mathematics and Computation},
  volume={441},
  pages={127698},
  year={2023},
  publisher={Elsevier}
}

@article{ciallella2024very,
  title={Very high order treatment of embedded curved boundaries in compressible flows: {ADER} discontinuous {G}alerkin with a space-time reconstruction for off-site data},
  author={Ciallella, Mirco and Clain, Stephane and Gaburro, Elena and Ricchiuto, Mario},
  journal={Computers \& Mathematics with Applications},
  volume={175},
  pages={1--18},
  year={2024},
  publisher={Elsevier}
}

@article{burman2014fictitious,
  title={Fictitious domain methods using cut elements: {III}. A stabilized {N}itsche method for {S}tokes' problem},
  author={Burman, Erik and Hansbo, Peter},
  journal={ESAIM: Mathematical Modelling and Numerical Analysis},
  volume={48},
  number={3},
  pages={859--874},
  year={2014}
}

@article{burman2015cutfem,
  title={{CutFEM}: discretizing geometry and partial differential equations},
  author={Burman, Erik and Claus, Susanne and Hansbo, Peter and Larson, Mats G and Massing, Andr{\'e}},
  journal={International Journal for Numerical Methods in Engineering},
  volume={104},
  number={7},
  pages={472--501},
  year={2015},
  publisher={Wiley Online Library}
}

@article{collins2023penalty,
  title={A penalty-free shifted boundary method of arbitrary order},
  author={Collins, J Haydel and Lozinski, Alexei and Scovazzi, Guglielmo},
  journal={Computer Methods in Applied Mechanics and Engineering},
  volume={417},
  pages={116301},
  year={2023},
  publisher={Elsevier}
}

@article{burman2025cut,
  title={Cut finite element methods},
  author={Burman, Erik and Hansbo, Peter and Larson, Mats G and Zahedi, Sara},
  journal={Acta Numerica},
  volume={34},
  pages={1--121},
  year={2025},
  publisher={Cambridge University Press}
}

@article{BOSCHERI2025114215,
title = {{High order treatment of moving curved boundaries: Arbitrary-{L}agrangian-Eulerian methods with a shifted boundary polynomial correction}},
journal = {Journal of Computational Physics},
volume = {539},
pages = {114215},
year = {2025},
author = {Walter Boscheri and Mirco Ciallella},
}

@article{boffi2003finite,
  title={A finite element approach for the immersed boundary method},
  author={Boffi, Daniele and Gastaldi, Lucia},
  journal={Computers \& Structures},
  volume={81},
  number={8-11},
  pages={491--501},
  year={2003},
  publisher={Elsevier}
}

@article{cottet2008eulerian,
  title={Eulerian formulation and level set models for incompressible fluid-structure interaction},
  author={Cottet, Georges-Henri and Maitre, Emmanuel and Milcent, Thomas},
  journal={ESAIM: Mathematical Modelling and Numerical Analysis},
  volume={42},
  number={3},
  pages={471--492},
  year={2008},
}

@article{leveque1994immersed,
  title={The immersed interface method for elliptic equations with discontinuous coefficients and singular sources},
  author={LeVeque, Randall J and Li, Zhilin},
  journal={SIAM Journal on Numerical Analysis},
  volume={31},
  number={4},
  pages={1019--1044},
  year={1994},
  publisher={SIAM}
}

@article{boffi2015finite,
  title={The finite element immersed boundary method with distributed {L}agrange multiplier},
  author={Boffi, Daniele and Cavallini, Nicola and Gastaldi, Lucia},
  journal={SIAM Journal on Numerical Analysis},
  volume={53},
  number={6},
  pages={2584--2604},
  year={2015},
  publisher={SIAM}
}

@article{peskin1977numerical,
  title={Numerical analysis of blood flow in the heart},
  author={Peskin, Charles S},
  journal={Journal of Computational Physics},
  volume={25},
  number={3},
  pages={220--252},
  year={1977},
  publisher={Elsevier}
}

@article{peskin2002immersed,
  title={The immersed boundary method},
  author={Peskin, Charles S},
  journal={Acta Numerica},
  volume={11},
  pages={479--517},
  year={2002},
  publisher={Cambridge University Press}
}

@article{strouboulis2000design,
  title={The design and analysis of the generalized finite element method},
  author={Strouboulis, Theofanis and Babu{\v{s}}ka, Ivo and Copps, Kevin},
  journal={Computer Methods in Applied Mechanics and Engineering},
  volume={181},
  number={1-3},
  pages={43--69},
  year={2000},
  publisher={Elsevier}
}

@article{mittal2005immersed,
  title={Immersed boundary methods},
  author={Mittal, Rajat and Iaccarino, Gianluca},
  journal={Annual Review of Fluid Mechanics},
  volume={37},
  number={1},
  pages={239--261},
  year={2005},
  publisher={Annual Reviews}
}

@article{hansbo2002unfitted,
  title={An unfitted finite element method, based on {N}itsche's method, for elliptic interface problems},
  author={Hansbo, Anita and Hansbo, Peter},
  journal={Computer Methods in Applied Mechanics and Engineering},
  volume={191},
  number={47-48},
  pages={5537--5552},
  year={2002},
  publisher={Elsevier}
}

@article{burman2010ghost,
  title={Ghost penalty},
  author={Burman, Erik},
  journal={Comptes Rendus. Math{\'e}matique},
  volume={348},
  number={21-22},
  pages={1217--1220},
  year={2010}
}

@article{monasse2012conservative,
  title={A conservative coupling algorithm between a compressible flow and a rigid body using an embedded boundary method},
  author={Monasse, Laurent and Daru, Virginie and Mariotti, Christian and Piperno, Serge and Tenaud, Christian},
  journal={Journal of Computational Physics},
  volume={231},
  number={7},
  pages={2977--2994},
  year={2012},
  publisher={Elsevier}
}

@article{atallah2022high,
  title={The high-order shifted boundary method and its analysis},
  author={Atallah, Nabil M and Canuto, Claudio and Scovazzi, Guglielmo},
  journal={Computer Methods in Applied Mechanics and Engineering},
  volume={394},
  pages={114885},
  year={2022},
  publisher={Elsevier}
}

@article{main2018shifted,
  title={The shifted boundary method for embedded domain computations. {P}art {I}: {P}oisson and {S}tokes problems},
  author={Main, Alex and Scovazzi, Guglielmo},
  journal={Journal of Computational Physics},
  volume={372},
  pages={972--995},
  year={2018},
  publisher={Elsevier}
}

@article{visbech2025spectral,
  title={A spectral element solution of the {P}oisson equation with shifted boundary polynomial corrections: influence of the surrogate to true boundary mapping and an asymptotically preserving Robin formulation},
  author={Visbech, Jens and Engsig-Karup, Allan P and Ricchiuto, Mario},
  journal={Journal of Scientific Computing},
  volume={102},
  number={1},
  pages={11},
  year={2025},
  publisher={Springer}
}

@article{atallah2025high,
  title={A high-order Shifted Interface Method for {L}agrangian shock hydrodynamics},
  author={Atallah, Nabil M and Mittal, Ketan and Scovazzi, Guglielmo and Tomov, Vladimir Z},
  journal={Journal of Computational Physics},
  volume={523},
  pages={113637},
  year={2025},
  publisher={Elsevier}
}

@article{song2018shifted,
  title={The shifted boundary method for hyperbolic systems: Embedded domain computations of linear waves and shallow water flows},
  author={Song, Ting and Main, Alex and Scovazzi, Guglielmo and Ricchiuto, Mario},
  journal={Journal of Computational Physics},
  volume={369},
  pages={45--79},
  year={2018},
  publisher={Elsevier}
}

@article{desmons2024fully,
  title={Fully {E}ulerian models for the numerical simulation of capsules with an elastic bulk nucleus},
  author={Desmons, Florian and Milcent, Thomas and Salsac, Anne-Virginie and Ciallella, Mirco},
  journal={Journal of Fluids and Structures},
  volume={127},
  pages={104109},
  year={2024},
  publisher={Elsevier}
}

@article{ciallella2025semi,
  title={Semi-implicit Eulerian method for the fluid structure interaction of elastic membranes},
  author={Ciallella, Mirco and Milcent, Thomas},
  journal={SIAM Journal on Scientific Computing},
  volume={47},
  number={3},
  pages={A1555--A1578},
  year={2025},
  publisher={SIAM}
}

@article{colomes2026generalized,
  title={The generalized shifted boundary method for geometry-parametric {PDE}s and time-dependent domains},
  author={Colom{\'e}s, Oriol and Modderman, Jan and Scovazzi, Guglielmo},
  journal={Computer Methods in Applied Mechanics and Engineering},
  volume={452},
  pages={118748},
  year={2026},
  publisher={Elsevier}
}

@article{gorgi2026isogeometric,
  title={The Isogeometric {G}ap-Shifted Boundary Method},
  author={Gorgi, Andrea and Antonelli, Nicol{\`o} and Cornejo, Alejandro and Scovazzi, Guglielmo and Rossi, Riccardo},
  journal={Available at SSRN 6179147},
  year={2026}
}

@article{antonelli2026isogeometric,
  title={Isogeometric multipatch coupling with arbitrary refinement and parametrization using the {G}ap--Shifted Boundary Method},
  author={Antonelli, Nicol{\`o} and Gorgi, Andrea and Zorrilla, Rub{\'e}n and Rossi, Riccardo},
  journal={Computer Methods in Applied Mechanics and Engineering},
  volume={456},
  pages={118913},
  year={2026},
  publisher={Elsevier}
}

@article{benzaken2024constructing,
  title={Constructing {N}itsche’s Method for Variational Problems},
  author={Benzaken, Joseph and Evans, John A and Tamstorf, Rasmus},
  journal={Archives of Computational Methods in Engineering},
  volume={31},
  number={4},
  pages={1867--1896},
  year={2024},
  publisher={Springer}
}

@article{juntunen2009nitsche,
  title={{N}itsche’s method for general boundary conditions},
  author={Juntunen, Mika and Stenberg, Rolf},
  journal={Mathematics of Computation},
  volume={78},
  number={267},
  pages={1353--1374},
  year={2009}
}

@article{Nitsche1971,
   author = {Joachim Nitsche},
   issue = {1},
   journal = {Abhandlungen aus dem mathematischen Seminar der Universität Hamburg},
   pages = {9-15},
   publisher = {Springer},
   title = {Über ein Variationsprinzip zur Lösung von {D}irichlet-Problemen bei Verwendung von Teilräumen, die keinen Randbedingungen unterworfen sind},
   volume = {36},
   year = {1971},
}

@article{ronquist1987legendre,
  title={A {L}egendre spectral element method for the {S}tefan problem},
  author={R{\o}nquist, Einar M and Patera, Anthony T},
  journal={International Journal for Numerical Methods in Engineering},
  volume={24},
  number={12},
  pages={2273--2299},
  year={1987},
  publisher={Wiley Online Library}
}

@book{karniadakis2005spectral,
    title = {{Spectral/hp element methods for computational fluid dynamics}},
    year = {2005},
    author = {Karniadakis, George and Sherwin, Spencer},
    publisher = {OUP Oxford}
}

@book{hesthaven2008nodal,
  title={Nodal discontinuous {G}alerkin methods: algorithms, analysis, and applications},
  author={Hesthaven, Jan S and Warburton, Tim},
  year={2008},
  publisher={Springer}
}

@article{yang2024optimal,
  title={Optimal surrogate boundary selection and scalability studies for the shifted boundary method on octree meshes},
  author={Yang, Cheng-Hau and Saurabh, Kumar and Scovazzi, Guglielmo and Canuto, Claudio and Krishnamurthy, Adarsh and Ganapathysubramanian, Baskar},
  journal={Computer Methods in Applied Mechanics and Engineering},
  volume={419},
  number={C},
  year={2024},
  publisher={Elsevier}
}

@article{ciallella2022extrapolated,
  title={Extrapolated {D}iscontinuity Tracking for complex {2D} shock interactions},
  author={Ciallella, Mirco and Ricchiuto, Mario and Paciorri, Renato and Bonfiglioli, Aldo},
  journal={Computer Methods in Applied Mechanics and Engineering},
  volume={391},
  pages={114543},
  year={2022},
  publisher={Elsevier}
}

@inproceedings{visbech2026recent,
  title={Recent progress on modeling nonlinear wave propagation and wave-structure interaction using a high-order shifted boundary method: Capabilities, challenges, and perspectives},
  author={Visbech, Jens and Engsig-Karup, Allan Peter and Bingham, Harry B and Ricchiuto, Mario},
  booktitle={Proceedings: The 41st International Workshop on Water Waves and Floating Bodies, Plitvice, Croatia},
  year={2026},
}

@article{visbech2025fnpf,
  title={{FNPF-SEM}: A parallel spectral element model in {F}iredrake for fully nonlinear water wave simulations},
  author={Visbech, Jens and Melander, Anders and Engsig-Karup, Allan Peter},
  journal={The International Journal of High Performance Computing Applications},
  pages={10943420261462394},
  year={2025},
  publisher={SAGE Publications Sage UK: London, England}
}

@article{engsigkarup2016stabilised,
    title = {{A stabilised nodal spectral element method for fully nonlinear water waves}},
    year = {2016},
    journal = {Journal of Computational Physics},
    author = {Engsig-Karup, A P and Eskilsson, C and Bigoni, D},
    pages = {1--21},
    volume = {318},
}

@article{pind2019time,
  title={Time domain room acoustic simulations using the spectral element method},
  author={Pind, Finnur and Engsig-Karup, Allan P and Jeong, Cheol-Ho and Hesthaven, Jan S and Mejling, Mikael S and Str{\o}mann-Andersen, Jakob},
  journal={The Journal of the Acoustical Society of America},
  volume={145},
  number={6},
  pages={3299--3310},
  year={2019},
  publisher={AIP Publishing}
}

@article{melander2025high,
  title={A High-Order Hybrid-Spectral Incompressible {N}avier-{S}tokes Model for Non-Linear Water Waves},
  author={Melander, Anders and Ebstrup Bitsch, Max and Chen, Dong and Peter Engsig-Karup, Allan},
  journal={International Journal for Numerical Methods in Fluids},
  volume={97},
  number={6},
  pages={1009--1021},
  year={2025},
  publisher={Wiley Online Library}
}

@article{carlier2023enriched,
  title={An enriched shifted boundary method to account for moving fronts},
  author={Carlier, Tiffanie and Nouveau, L{\'e}o and Beaugendre, Heloise and Colin, Mathieu and Ricchiuto, Mario},
  journal={Journal of Computational Physics},
  volume={489},
  pages={112295},
  year={2023},
  publisher={Elsevier}
}


\end{document}